\documentclass{amsart}

\usepackage{lipsum}
\usepackage{amsfonts,bm,amssymb,amsmath}
\usepackage{graphicx}
\usepackage{epstopdf}
\usepackage[caption=false]{subfig}
\usepackage{pgfplots}
\usepackage{algorithmic}
\usepackage{amsopn}
\usepackage{amsrefs}
\usepackage{float}

\usepackage{hyperref}
\usepackage{cleveref}
\usepackage{geometry}
\usepackage{color}
\usepackage{diagbox}

\newtheorem{thm}{Theorem}[section]
\newtheorem{definition}{Definition}[section]
\newtheorem{lemma}{Lemma}[section]
\newtheorem{remark}{Remark}[section]

\hypersetup{
	colorlinks,
	linkcolor={red!50!black},
	citecolor={blue!50!black},
	urlcolor={blue!80!black}
}

\numberwithin{equation}{section}

\definecolor{newcolor1}{rgb}{.8,.349,.1}
\colorlet{bblue}{blue!50!black}
\crefformat{equation}{(#2#1#3)}
\crefmultiformat{equation}{(#2#1#3)}{ and~(#2#1#3)}{, (#2#1#3)}{ and~(#2#1#3)}

\crefformat{figure}{Figure~#2#1#3}
\crefmultiformat{figure}{Figures~ #2#1#3}{ and~#2#1#3}{, (#2#1#3)}{ and~(#2#1#3)}
\crefformat{table}{Table~#2#1#3}

\def\e{\mbox{\boldmath $e$}}
\def\f{\mbox{\boldmath $f$}}

\def\g{\mbox{\boldmath $g$}}

\def\m{\mbox{\boldmath $m$}}
\def\n{\mbox{\boldmath $n$}}

\def\x{\mbox{\boldmath $x$}}

\newcommand\dt {{\Delta t}}
\def\0{\mbox{\boldmath $0$}}
\def\um{\underline{\bm m}}

\def\ppsi{\mbox{\boldmath $\psi$}}

\begin{document}

\title[A third order accurate linear scheme for LLG Equation]{Error Analysis of Third-Order in Time and Fourth-Order Linear Finite Difference Scheme for Landau-Lifshitz-Gilbert Equation under Large Damping Parameters}





\author{Changjian Xie}
\address{School of Mathematics and Physics\\ Xi'an Jiaotong-Liverpool University\\Ren'ai Rd. 111, Suzhou, 215123\\ Jiangsu\\ China.}
\email{Changjian.Xie@xjtlu.edu.cn}

\author{Cheng Wang}
\address{Mathematics Department\\ University of Massachusetts\\ North Dartmouth\\ MA 02747\\ USA.}
\email{cwang1@umassd.edu}


\date{\today}



\begin{abstract}
This work proposes and analyzes a fully discrete numerical scheme for solving the Landau-Lifshitz-Gilbert (LLG) equation, which achieves fourth-order spatial accuracy and third-order temporal accuracy.Spatially, fourth-order accuracy is attained through the adoption of a long-stencil finite difference method, while boundary extrapolation is executed by leveraging a higher-order Taylor expansion to ensure consistency at domain boundaries. Temporally, the scheme is constructed based on the third-order backward differentiation formula (BDF3), with implicit discretization applied to the linear diffusion term for numerical stability and explicit extrapolation employed for nonlinear terms to balance computational efficiency. Notably, this numerical method inherently preserves the normalization constraint of the LLG equation, a key physical property of the system.Theoretical analysis confirms that the proposed scheme exhibits optimal convergence rates under the \(\ell^{\infty}([0,T],\ell^2)\) and \(\ell^2([0,T],H_h^1)\) norms. Finally, numerical experiments are conducted to validate the correctness of the theoretical convergence results, demonstrating good agreement between numerical observations and analytical conclusions.

\end{abstract}

\maketitle

\section{Introduction}

The Landau-Lifshitz-Gilbert (LLG) equation is given by 
\begin{align}\label{c1}
{\m}_t=-{\m}\times\Delta{\m}-\alpha{\m}\times({\m}\times\Delta{\m}) , 
\end{align}
with
\begin{equation}\label{boundary}
\frac{\partial{\m}}{\partial \boldmath {\nu}}\Big|_{\Gamma}=0,
\end{equation}
where $\Gamma = \partial \Omega$ and $\boldmath {\nu}$ is the unit outward normal vector along $\Gamma$. Here ${\m}\,:\,\Omega\subset\mathbb{R}^d\to S^2$ represents the magnetization vector field with $|{\m}|=1,\;\forall x\in\Omega$, $d=1,2,3$ is the spatial dimension, and $\alpha>0$ is the damping parameter. The first term on the right hand side of \cref{c1} is the gyromagnetic term, and the second term is the damping term. Compared to the original LLG equation \cite{Landau2015On}, \cref{c1} only includes the exchange term which poses the main difficulty in numerical analysis, as done in the literature \cite{weinan2001numerical, cimrak2004iterative, bartels2006convergence, gao2014optimal}. 
To ease the presentation, we set $\Omega=[0, 1]^d$, in which $d$ is the dimension. 

In the case of a large dampening parameter $\alpha$, so that the term of $-\alpha{\m}\times({\m}\times\Delta{\m})$ is more dominant, we make the following observation: 
\begin{align*} 
- {\m}\times({\m}\times\Delta{\m}) = \Delta \m + | \nabla \m |^2 \m ,  \quad 
\mbox{since $| \m | \equiv 1$} . 
\end{align*}
In turn, the LL equation~\eqref{c1} could be rewritten as 
\begin{align} \label{equation-LL-alt}
{\m}_t=-{\m}\times\Delta{\m} + \alpha \Delta \m + \alpha | \nabla \m |^2 \m . 
\end{align}
Again, the homogeneous Neumann boundary condition~\eqref{boundary} is imposed. 

We have proposed an effective method to handle the problem of large dissipation coefficients in previous work \cite{Cai2022JCP, Cai2023MMAS}, where we adopted the second-order Backward Differentiation Formula (BDF2) and treated the linear diffusion term $\alpha \Delta \m$ implicitly, while the two nonlinear terms, namely $-\m\times \Delta \m$ and $\alpha |\nabla \m|^2\m$ are discretized in a fully explicit way, where the high order term using the interpolation of pre-projection solution and low order nonlinear terms using the interpolation of post-projection solution. Subsequently, a point-wise projection is applied to the intermediate field, so that the numerical solution of $\m$ has a unit length at the point-wise level. For the approximation of $\partial_t \m$ using BDF2, we adopt the pre-projection solution for all steps. Such a numerical approach leads to a linear system with constant coefficients independent the updated magnetization at each time step. Based on this subtle fact, the linear numerical scheme has demonstrated great advantages in the simulation of ferromagnetic materials for large damping parameters. 

In this work, we aim to maximize the advantages of such efficiency and accuracy. Therefore, we propose a feasible numerical method based on the idea of third order BDF. We provide the convergence analysis and the optimal rate error estimate for the proposed linear numerical scheme, in the discrete $\ell^{\infty}([0,T];\ell^2)\cup\ell^2([0,T];H^1_h)$ norm, if the damping parameter is greater than $7$. The proof of stability estimate of the projection step is clear as previous work, which plays a crucial role in the rigorous error estimate for the original error function. 

The The rest of this paper is organized as follows. The fully discrete numerical schemes and state the main theoretical result of convergence are reviewed in \Cref{sec:main theory}. The detailed proof is provided in \Cref{Sec:proof}. Some numerical results are presented in \Cref{sec:experiments}. Finally, some concluding remarks are made in \Cref{sec:conclusions}.

\section{Main theoretical results}
\label{sec:main theory}

\subsection{Fully discretization}\label{discretisations}

The finite difference method is used to approximate \cref{c1} and \cref{boundary}. Denote the spatial step-size by $h$ in the 1D case and divide $[0,1]$ into $N_x$ equal segments. 
Define $x_i=ih$, $i=0,1,2,\cdots,N_x$, with $x_0=0$, $x_{N_x}=1$, and $\hat{x}_i=x_{i-\frac{1}{2}}=(i-\frac{1}{2})h$, $i=1,\cdots,N_x$.  Denote the magnetization obtained by the numerical scheme at $(\hat{x}_i,t^n)$ by $\m_i^n$, in which we have introduced $t^n=nk$, with $k$ being the temporal step-size, and $n\leq \left\lfloor\frac{T}{k}\right\rfloor$, $T$ being the final time. To approximate the boundary condition \cref{boundary}, we introduce ghost points $x_{-\frac32}, x_{-\frac{1}{2}}, x_{N_x+\frac{1}{2}}, x_{N_x+\frac32}$ and apply Taylor expansions for $x_{-\frac32}$, $x_{-\frac{1}{2}}$, $x_{\frac{1}{2}}$, $x_{\frac32}$ at $x_0$, and $x_{N_x+\frac{3}{2}}$,  $x_{N_x+\frac{1}{2}}$, $x_{N_x-\frac{1}{2}}$, $x_{N_x+\frac{3}{2}}$ at $x_{N_x}$, respectively.  We then obtain a third order extrapolation formula:
\[
\m_0=\m_1, \, \, \m_{-1} = \m_2 , \quad \m_{N_x+1}=\m_{N_x} , \, \,  \m_{N_x+2}=\m_{N_x-1} . 
\]
In the 3D case, we have spatial stepsizes $h_x=\frac{1}{N_x}$, $h_y=\frac{1}{N_y}$, $h_z=\frac{1}{N_z}$ and grid points $(\hat{x}_i,\hat{y}_j,\hat{z}_k)$, with $\hat{x}_i=x_{i-\frac{1}{2}}=(i-\frac{1}{2})h_x$, $\hat{y}_j=y_{j-\frac{1}{2}}=(j-\frac{1}{2})h_y$ and $\hat{z}_k=z_{k-\frac{1}{2}}=(k-\frac{1}{2})h_z$ ( $ 0 \le i \le N_x+1$, $0 \le j \le N_y+1$, $0 \le k \le N_z+1$). The extrapolation formula along the $z$ direction near $z=0$ and $z=1$ is
\begin{equation}\label{BC-1}
\m_{i,j,0}=\m_{i,j,1}, \, \, \m_{i,j,-1}=\m_{i,j,2} , \, \,  
\m_{i,j,N_z+1}=\m_{i,j,N_z} , \, \, \m_{i,j,N_z+2}=\m_{i,j,N_z -1} . 
\end{equation}
Extrapolation formulas for the boundary condition along other directions can be derived similarly.

In addition, given $\m_e (\cdot, t=0)$ as the exact initial data at $t=0$ (with $\m_e$ the exact solution), the numerical initial data for $\m$ is set as  
\begin{equation} 
  \m_{i,j,k}^0 = {\mathcal P}_h \m_e (\hat{x}_i, \hat{y}_j, \hat{z}_k, t=0) ,  \quad 
  \mbox{${\mathcal P}_h$ is the point-wise interpolation.}   \label{initial data-1} 
\end{equation} 

In addition, to obtain a fourth order spatial accuracy, the following long-stencil difference operators are introduced, to approximate $\partial_x$, $\partial_x^2$, respectively:  
\begin{eqnarray} 
  \hspace{-0.35in}  
  {\mathcal D}_{x,(4)}^1 f_{i,j,k} &=& \tilde{D}_x ( 1 - \frac{h^2}{6} D_x^2 ) f_{i,j,k} \nonumber 
\\
  &=& 
   \frac{  f_{i-2,j,k} - 8 f_{i-1,j,k}  + 8 f_{i+1,j,k} - f_{i+2,j,k} }{12 h} ,  
  \label{FD-4th-1} 
\\
  \hspace{-0.35in}  
   {\mathcal D}_{x,(4)}^2 f_{i,j,k} &=& D_x^2 ( 1 - \frac{h^2}{12} D_x^2 ) f_{i,j,k}  \nonumber 
\\
  &=& 
    \frac{ - f_{i-2,j,k} + 16 f_{i-1,j,k,k} - 30 f_{i,j,k} + 16 f_{i+1,j,k} - f_{i+2,j,k} }{12 h^2 } . 
   \label{FD-4th-2} 
\end{eqnarray} 
The long-stencil difference operators in the $y$ and $z$ directions, namely, ${\mathcal D}_{y,(4)}^1$, ${\mathcal D}_{y,(4)}^2$, ${\mathcal D}_{z,(4)}^1$, ${\mathcal D}_{z,(4)}^2$, could be defined in a similar fashion. In turn, we denote $\Delta_{h,(4)} = {\mathcal D}_{x,(4)}^2 + D_{y, (4)}^2 + D_{z, (4)}^2$. 

Denote the temporal step-size by $k$, and define $t^n=nk$, $n\leq \left\lfloor\frac{T}{k}\right\rfloor$ with $T$ the final time. The third-order BDF approximation is applied to the temporal derivative:
\begin{equation*}  \label{BDF2-1}
\frac{\frac{11}{6} \m_h^{n+3} - 3 \m_h^{n+2}+\frac32 \m_h^{n+1} - \frac13 \m_h^n}{k} = \frac{\partial}{\partial t}\m_h^{n+3} + \mathcal{O}(k^3) . 	 	
\end{equation*}
Note that the right hand side of the above equation is evaluated at $t^{n+3}$, a direct application of the BDF method leads to a fully nonlinear scheme. 
To overcome this subtle difficulty, we make use of the alternate PDE formulation~\eqref{equation-LL-alt}, treat the linear diffusion term $\alpha \Delta \m$ implicitly, and both nonlinear terms, namely $- \m \times \Delta \m$ and $\alpha | \nabla \m |^2 \m$, in a fully explicit way. Afterward, a point-wise projection is applied to the intermediate field, so that the numerical solution of $\m$ has a unit length at a point-wise level. 
In more details, the following numerical scheme is proposed: 
\begin{align}
\hat{\m}_h^{n+3} &= 3 \m_h^{n+2} - 3 \m_h^{n+1} + \m_h^n , \, \, \, 
\hat{\tilde{\m}}_h^{n+3} = 3 \tilde{\m}_h^{n+2} - 3 \tilde{\m}_h^{n+1} + \tilde{\m}_h^n , 
\label{cc} 
\end{align}

\begin{align}\label{scheme-1-1}
  & 
\frac{\frac{11}{6} \tilde{\m}_h^{n+3} - 3 \tilde{\m}_h^{n+2} + \frac32 \tilde{\m}_h^{n+1} 
- \frac13 \tilde{\m}_h^n }{k} \\ 
= &  - \hat{\m}_h^{n+3} \times \Delta_{h, (4)} \hat{\tilde{\m}}_h^{n+3} 
 + \alpha \Delta_{h, (4)} \tilde{\m}_h^{n+3}  
 + \alpha | \tilde{\nabla}_{h, (4)} \hat{\m}_h^{n+3} |^2 \hat{\m}_h^{n+3} , \nonumber
\end{align}
\begin{align}
\m_h^{n+3} &= \frac{\tilde{\m}_h^{n+3}}{ |\tilde{\m}_h^{n+3}| } , \label{scheme-1-2}
\end{align}
in which $| \tilde{\nabla}_{h, (4)} \f_h |^2$ is defined as (for $\f_h = (f_1, f_2, f_3 )^T)$:  
   \begin{equation} 
   | \tilde{\nabla}_{h, (4)} \f_h |^2 = \sum_{\ell=1}^3 \Big( 
   ( {\mathcal D}_{x,(4)}^1f_\ell )^2  + ( {\mathcal D}_{y,(4)}^1f_\ell )^2 
   + ( {\mathcal D}_{z,(4)}^1f_\ell )^2 \Big) . \label{defi-gradient-4th} 
\end{equation} 
The discrete boundary condition~\eqref{BC-1} is imposed for $\tilde{\m}_h^{n+3}$ in~\eqref{scheme-1-1}. In fact, this boundary condition could be rewritten as $(\nabla_h \tilde{\m}_h^{n+3} \cdot \n ) \mid_{\partial \Omega} =0$.

\begin{remark}
To kick start the iteration of, we can use the second-order semi-implicit projection scheme, and the numerical method is still third-order accurate.
\end{remark}
	
\begin{remark} 
With both nonlinear terms treated fully explicitly in the numerical scheme~\eqref{scheme-1-1}, this approach would greatly improve the computational efficiency, since only a Poisson equation needs to be solved at each time step.  
\end{remark}

\subsection{Main theorem}

For simplicity of presentation, we assume that $N_x = N_y =N_z=N$ so that $h_x = h_y = h_z =h$. An extension to the general case is straightforward. In the finite difference approximation, all the numerical values are assigned on the numerical grid points. As a result, the discrete grid functions (with notations $\f_h$, $\g_h$), which are only defined over the corresponding numerical grid points, are introduced.    

First, we introduce the discrete $\ell^2$ inner product and discrete $\| \cdot \|_2$ norm.
\begin{definition}[Inner product and $\| \cdot \|_2$ norm]
	For grid functions $\f_h$ and $\g_h$ over the uniform numerical grid, we define
	\begin{align}
	\langle {\bm f}_h,{\bm g}_h \rangle = h^d\sum_{\mathcal{I}\in \Lambda_d} \f_{\mathcal{I}}\cdot\g_{\mathcal{I}},  \label{defi-inner product-1} 
	\end{align}
	where $\Lambda_d$ is the index set and $\mathcal{I}$ is the index which closely depends on $d$.
	In turn, the discrete $\| \cdot \|_2$ norm is given by
	\begin{equation}
	\| {\bm f}_h \|_2 =  ( \langle {\bm f}_h,{\bm f}_h \rangle )^{1/2} . \label{defi-L2 norm}
	\end{equation}
	In addition, the discrete $H_h^1$-norm is given by $\| \f_h \|_{H_h^1}^2 :=\|\f_h\|_2^2+\|\nabla_h \f_h\|_2^2$.
\end{definition}

\begin{definition}[Discrete $\| \cdot \|_\infty$ norm]
	For the grid function $\f_h$ over the uniform numerical grid, we define
	\begin{align*}
	\| \f_h \|_{\infty} = \max_{\mathcal{I}\in\Lambda_d}\|\f_{\mathcal{I}} \|_{\infty} .
	\end{align*}
\end{definition}

\begin{definition}
	For the grid function $\f_h$, we define the average of summation as
	\begin{align*}
	\overline{\f}_h=h^d\sum_{\mathcal{I}\in \Lambda_d}\f_{\mathcal{I}}.
	\end{align*}
\end{definition}

\begin{definition}
       For any grid function $\f_h$ with $\overline{\f_h}=0$, a discrete inverse Laplacian operator is defined as: $\ppsi_h = (-\Delta_{h, (4)})^{-1} \f_h$ is the unique grid function satisfying  
\begin{equation*} 
     - \Delta_{h, (4)} \ppsi_h =  \f_h ,  \quad  
     (\nabla_h \ppsi_h \cdot \n ) \mid_{\partial \Omega} =0  , \quad    
       \overline{\ppsi_h} = 0 . 
\end{equation*} 
It is noticed that the zero-average constraint, $\overline{\ppsi_h}=0$, makes the operator $(-\Delta_{h, (4)})^{-1}$ uniquely defined. In turn, 
a discrete $H_h^{-1}$-norm is introduced for any $\f_h$ with $\overline{\f_h}=0$:  
	\begin{equation*}
	\|\f_h\|_{-1}^2=\langle(-\Delta_{h, (4)})^{-1}\f_h,\f_h\rangle.
	\end{equation*}
\end{definition}

The unique solvability analysis of scheme~\cref{cc}-\cref{scheme-1-2} is based on a rewritten form of \cref{scheme-1-1}:  
\begin{equation*} 
\begin{aligned} 
 \Big( \frac{11}{6 k} I - \alpha \Delta_{h, (4)} \Big) \tilde{\m}_h^{n+3} 
 = \f_h^n := & \frac{3 \tilde{\m}_h^{n+2} - \frac32 \tilde{\m}_h^{n+1} + \frac13 \tilde{\m}_h^n }{k} 
\\
  & 
  - \hat{\m}_h^{n+3} \times \Delta_h \hat{\m}_h^{n+3}  
 + \alpha | \tilde{\nabla}_{h, (4)} \hat{\m}^{n+3} |^2 \hat{\m}^{n+3} .  
\end{aligned} 
\end{equation*} 
The left hand side corresponds to a positive-definite symmetric operator, with discrete Fourier Cosine transformation could be very efficiently applied. As a result, its unique solvability is obvious.

The main theoretical result is the optimal rate convergence analysis. 

\begin{thm}\label{cccthm2} Let $\m_e \in C^4 ([0,T]; C^0) \cap C^3([0,T]; C^1) \cap L^{\infty}([0,T]; C^6)$ be the exact solution of \cref{c1} with the initial data $\m_e ({\x},0)=\m_e ^0({\x})$ and ${\m}_h$ be the numerical solution of the equation~\cref{cc}-\cref{scheme-1-2} with the initial data ${\m}_h^0=\m_{e,_h}^0$, $\m_h^1= \m _{e,h}^1$ and $\m_h^2= \m _{e,h}^2$. Suppose that the initial error satisfies $\|\m_{e,h}^\ell - \m_h^\ell \|_2 +\|\nabla_h ( \m_{e,h}^\ell - \m_h^\ell ) \|_2 = \mathcal{O} (k^3 + h^4),\,\ell=0,1, 2$, and $k\leq \mathcal{C}h$. In addition, we assume that $\alpha > 7$. Then the following convergence result holds as $h$ and $k$ goes to zero:
	\begin{align} \label{convergence-0} 
	\| \m_{e,h}^n - \m_h^n \|_{2}+ \|\nabla_h ( \m_{e,h}^n - \m_h^n ) \|_{2} &\leq \mathcal{C}(k^3+h^4) , \quad \forall n \ge 3 ,
	\end{align}	
	in which the constant $\mathcal{C}>0$ is independent of $k$ and $h$.
\end{thm}

\subsection{A few preliminary estimates} 
The proof of the standard inverse inequality and discrete Gronwall inequality could be obtained in existing textbooks; we just cite the results here. The inverse inequality presented in~\cite{Ciarlet1978} is in the finite element version; its extension to the finite difference version is straightforward.  
we just cite the results here.
\begin{lemma}(Inverse inequality) \cite{Ciarlet1978} \label{ccclemC1}.
	The inverse inequality implies that
	\begin{align*}\label{ccc39}
	\|{\e}_h^{n}\|_{\infty} \leq \gamma {h}^{-d/2}\|{\e}_h^{n}\|_2 ,  \quad
	\|\nabla_h{\e}_h^{n}\|_{\infty} \leq \gamma {h}^{-d/2}\|\nabla_h{\e}_h^{n}\|_2 , \nonumber
	\end{align*}
in which constant $\gamma$ depends on the form of the discrete $\| \cdot \|_2$ norm. Under the definition~\eqref{defi-inner product-1} and \eqref{defi-L2 norm} for the cell-centered grid function, such a constant could be taken as $\gamma =1$.  
\end{lemma} 

\begin{lemma}(Discrete Gronwall inequality) \cite{Girault1986} \label{ccclem1}. Let $\{\alpha_j\}_{j\geq 0}$, $\{\beta_j\}_{j\geq 0}$ and $\{\omega_j\}_{j\geq 0}$ be sequences of real numbers such that
	\begin{equation*}
	\alpha_j\leq \alpha_{j+1},\quad \beta_j\geq 0,\quad and \quad \omega_j\leq \alpha_j+\sum_{i=0}^{j-1}\beta_i\omega_i , \quad \forall j \geq 0.
	\end{equation*}
	Then it holds that
	\begin{equation*}
	\omega_j\leq \alpha_j\exp\left\{\sum_{i=0}^{j-1}\beta_i\right\} , \quad \forall j \geq 0.
	\end{equation*}
\end{lemma}

\begin{lemma}[Summation by parts]\label{summation}
	For any grid functions $\f_h$ and $\g_h$ satisfying the discrete boundary condition~\cref{BC-1}, the following identities are valid:
	\begin{align} 
	\left\langle -\Delta_h \f_h,\g_h\right\rangle = & \left\langle \nabla_h \f_h,\nabla_h \g_h\right\rangle ,  \label{sum1}	 
\\
   	\langle D_x^4 \f_h,\g_h \rangle = & 
	\langle D_x^2 \f_h, D_x^2 \g_h \rangle ,  \, \, 
	\langle D_y^4 \f_h,\g_h \rangle =  
	\langle D_y^2 \f_h, D_y^2 \g_h \rangle  ,    \nonumber 
\\ 
	\langle D_z^4 \f_h,\g_h \rangle =  & 
	\langle D_z^2 \f_h, D_z^2 \g_h \rangle	 ,  \label{sum2}	 	
\\
    \langle - \Delta_{h, (4)} \f_h , \g_h \rangle 
    = & \langle \nabla_{h, (4)} \f_h , \nabla_{h, (4)} \g_h \rangle  \nonumber 
\\
    :=  & \langle \nabla_h \f_h, \nabla_h \g_h \rangle   
    + \frac{h^2}{12} \langle \Delta_h \f_h, \Delta_h \g_h \rangle .  \label{sum3}	 
	\end{align} 
In turn, for any discrete grid function $\f_h$, the following norm could be introduced: 
\begin{equation} 
\begin{aligned} 
  \| \nabla_{h, (4)} \f_h \|_2 =: & 
  \Big( \langle \nabla_{h, (4)} \f_h , \nabla_{h, (4)} \f_h \rangle \Big)^\frac12  
\\
  =& 
   \Big( \| \nabla_h \f_h \|_2^2    
    + \frac{h^2}{12} ( \| D_x^2 \f_h \|_2^2 + \| D_y^2 \f_h \|_2^2 
    + \| D_z^2 \f_h \|_2^2 )  \Big)^\frac12  . 
\end{aligned} 
    \label{defi-gradient-4th-L2 norm} 
\end{equation} 
\end{lemma}

\begin{proof} 
The standard summation by parts formula~\eqref{sum1} has been proved in a recent work~\cite{chen20}. The identities in~\eqref{sum2} could be proved in the same manner, and the technical details are skipped for the sake of brevity. In turn, the formula~\eqref{sum3} is a direct consequence of the long-stencil difference definition~\eqref{FD-4th-2}, combined with~\eqref{sum1}, \eqref{sum2}. The proof of Lemma~\ref{summation} is complete.  
\end{proof} 

In addition, the following $\| \cdot \|_2$ bound estimates between $\nabla_h$, $\nabla_{h, (4)}$ and $\tilde{\nabla}_{h, (4)}$ (as introduced in~\eqref{defi-gradient-4th}) in the convergence analysis. The detailed proof has been provided in a recent article. 

\begin{lemma}[$\| \cdot \|_2$ bound for different gradient operators] \label{gradient bound}
	For any grid functions $\f_h$ satisfying the discrete boundary condition~\cref{BC-1}, the following inequalities are valid:
	\begin{align} 
	\| \nabla_h \f_h \|_2 \le & \| \nabla_{h, (4)} \f_h \|_2 
	\le   \frac{2}{\sqrt{3}} \| \nabla_h \f_h \|_2,  \label{gradient-L2-0-1}	 
\\
    \| \tilde{\nabla}_{h, (4)} \f_h \|_2 \le & \frac53 \| \nabla_h \f_h \|_2 ,   \, \, 
    \tilde{\nabla}_{h, (4)} \f_h= ( {\mathcal D}_{x,(4)}^1 \f_h ,  {\mathcal D}_{y,(4)}^1 \f_h ,  
  {\mathcal D}_{z,(4)}^1 \f_h ) .     
    \label{gradient-L2-0-2}        	 	
	\end{align} 
\end{lemma}

The following estimate will be utilized in the convergence analysis; its proof has been provided in a recent article~\cite{chen20}. In the sequel, for simplicity of our notation, we will use the uniform constant $\mathcal{C}$ to denote all the controllable constants in this paper.

\begin{lemma}[Discrete gradient acting on cross product] \cite{chen20} \label{lem27}
	For grid functions $\f_h$ and $\g_h$ over the uniform numerical grid, we have
	\begin{align}
	\|\nabla_{h, (4)} ({\f}_h \times{\g}_h ) \|_2^2 
	\le & \frac43 \|\nabla_h({\f}_h \times{\g}_h ) \|_2^2 \nonumber \\ 
	 \leq & \mathcal{C}\Big(\|{\f}_h\|_{\infty}^2\cdot \|\nabla_h {\g}_h\|_2^2+\|{\g}_h\|_{\infty}^2\cdot \|\nabla_h{\f}_h\|_2^2\Big),\label{lem27_1}\\
	\left\langle \f_h\times \Delta_{h, (4)} \g_h , \hat{\g}_h \right\rangle = & \left\langle \hat{\g}_h \times \f_h ,\Delta_{h, (4)} \g_h\right\rangle\label{lem27_2} .
	\end{align}
\end{lemma}

The following estimate will be used in the error estimate at the projection step; its proof has been provided in a recent article~\cite{chen20}.

\begin{lemma}  \cite{chen20} \label{lem 6-0}
	Consider $\um_h = {\mathcal P}_h \m_e$, in which ${\mathcal P}_h$ stands for the point-wise interpolation of a continuous function over the numerical grid points, the continuous function $\m_e$ satisfies a regularity requirement $\| \m_e \|_{W^{1, \infty}} \le {\mathcal C}$, 
$| \m_e | = 1$ at a point-wise level. 
For any numerical solution $\tilde{\m}_h$, we define $\m_h = \frac{\tilde{\m}_h}{ | \tilde{\m}_h |}$. Suppose both numerical profiles satisfy the following $W_h^{1, \infty}$ bounds
	\begin{align}
	& |\tilde{\m}_h| \ge \frac12 ,  \quad \mbox{at a pointwise level} ,  \label{lem 6-1-0}
	\\
	& \| \m_h \|_{\infty} + \| \nabla_h \m_h \|_\infty \le M ,  \quad \| \tilde{\m}_h \|_{\infty} + \| \nabla_h \tilde{\m}_h \|_\infty \le M,  \label{lem 6-1}
	\end{align}
	and we denote the numerical error functions as $\e_h = \m_h - \um_h$, $\tilde{\e}_h = \tilde{\m}_h - \um_h$. Then the following estimate is valid
	\begin{equation}
	\| \e_h \|_2 \le 2 \| \tilde{\e}_h \|_2  ,  \quad
	\| \nabla_h \e_h \|_2 \le \mathcal{C} ( \| \nabla_h \tilde{\e}_h \|_2
	+ \| \tilde{\e}_h \|_2 )  . \label{lem 6-2}
	\end{equation}
\end{lemma}

In addition, to establish the optimal rate convergence analysis, we have to recall the telescope formula in \cite{JLiu2013} for the third order BDF temporal discretization operator in the following lemma; also see~\cite{Hao2020, yao17} for the related discussion. 

\begin{lemma} \label{lem:3rd order BDF}
For the third order BDF temporal discretization operator, there exists $\alpha_i$, $i=1,\cdots,10$, $\alpha_1 \ne 0$, such that
\begin{eqnarray} \label{BDF-3-est-0}
&& \langle \frac{11}{6} e^{n+1} - 3 e^n + \frac32 e^{n-1} - \frac13 e^{n-2} ,
 2 e^{n+1} - e^n \rangle \\
&=&  \| \alpha_1 e^{n+1} \|_2^2 - \| \alpha_1 e^n \|_2^2
 + \| \alpha_2 e^{n+1} + \alpha_3 e^n \|_2^2
 - \| \alpha_2 e^n + \alpha_3 e^{n-1} \|_2 ^2 \nonumber
\\
  &&
  + \| \alpha_4 e^{n+1} + \alpha_5 e^n + \alpha_6 e^{n-1} \|_2^2
 - \| \alpha_4 e^n + \alpha_5 e^{n-1} + \alpha_6 e^{n-2} \|_2^2 \nonumber
\\
  &&+ \| \alpha_7 e^{n+1} + \alpha_8 e^n + \alpha_9 e^{n-1}
   + \alpha_{10} e^{n-2} \|_2^2 . \nonumber
\end{eqnarray}
\end{lemma}

\section{The optimal rate convergence analysis: Proof of Theorem~\ref{cccthm2}}\label{Sec:proof} 

\begin{proof}
	First, denote $\underline{\m} = \m_e$. Subsequently, we extend the profile $\underline{\m}$ to the numerical ``ghost" points, according to the extrapolation formula~\cref{BC-1}:
	\begin{equation}
	\underline{\m}_{i,j,0}= \underline{\m}_{i,j,1} , \quad
	\underline{\m}_{i,j,N_z+1} = \underline{\m}_{i,j,N_z} ,  \label{exact-3}
	\end{equation}
and the extrapolation for other boundaries can be formulated in the same manner. Next, we prove that such an extrapolation yields a higher order $\mathcal{O}(h^5)$ approximation, instead of the standard $\mathcal{O}(h^3)$ accuracy. Also see the related works~\cite{STWW2003, Wang2000, Wang2004} in the existing literature. 


	Performing a careful Taylor expansion for the exact solution around the boundary section $z=0$, combined with the mesh point values: $\hat{z}_0 = - \frac12 h$, $\hat{z}_1 = \frac12 h$, we get
	\begin{align}
	\m_e  (\hat{x}_i, \hat{y}_j, \hat{z}_0 )
	&= \m_e (\hat{x}_i, \hat{y}_j, \hat{z}_1 )
	- h \partial_z \m_e (\hat{x}_i, \hat{y}_j, 0 )
	- \frac{h^3}{24} \partial_z^3 \m_e (\hat{x}_i, \hat{y}_j, 0 )
	+   \mathcal{O}(h^5) \nonumber
	\\
	&= \m_e (\hat{x}_i, \hat{y}_j, \hat{z}_1 )
	- \frac{h^3}{24} \partial_z^3 \m_e (\hat{x}_i, \hat{y}_j, 0 )
	+   \mathcal{O}(h^5) ,  \label{exact-4}
	\end{align}
	in which the homogenous boundary condition has been applied in the second step. It remains to determine $\partial_z^3 \partial_z \m_e (\hat{x}_i, \hat{y}_j, 0 )$, for which we use information from the rewritten PDE~\eqref{equation-LL-alt} and its derivatives. Applying $\partial_z$ to the first evolutionary equation in~\eqref{equation-LL-alt} along the boundary section $\Gamma_z: z=0$ gives
\begin{eqnarray}
\begin{aligned} 
  & 
  (m_1)_{zt} 
  - 2 \alpha ( m_1 ( \nabla m_1 \cdot \nabla (m_1)_z  + \nabla m_2 \cdot \nabla (m_2)_z
  + \nabla m_3 \cdot \nabla (m_3)_z )  ) 
\\
  & 
  - \alpha | \nabla \m_e |^2 (m_1)_z  
   - \alpha ( (m_1)_{zxx} + (m_1)_{zyy} + \partial_z^3 m_1 )  
\\
  =& 
  ( m_3 )_z \Delta m_2  
  + m_3 ( (m_2)_{zxx} + (m_2)_{zyy} + \partial_z^3 m_2 )      
\\
  &  
  -  ( m_2 )_z \Delta m_3 
  - m_2  ( (m_3)_{zxx} + (m_3)_{zyy} + \partial_z^3 m_3 )  , \quad
 \mbox{on} \, \, \Gamma_z .  
\end{aligned} 
 \label{scheme-BC-2}
\end{eqnarray}
The first, third terms, and the first two parts in the fourth term on the left-hand side of~\eqref{scheme-BC-2} disappear, due to the homogeneous Neumann boundary condition for $m_1$. For the second term on the left hand side, we observe that
\begin{equation}
   \nabla m_1 \cdot \nabla (m_1)_z  = (m_1)_x \cdot (m_1)_{zx} + (m_1)_y \cdot (m_1)_{zy} 
   + (m_1)_z \cdot (m_1)_{zz}
   = 0 ,   \quad \mbox{on} \, \, \Gamma_z ,  \label{scheme-BC-3}
\end{equation}
since $(m_1)_z =0$ on the boundary section. Similar derivations could be made to the two other terms on the left hand side:
\begin{equation}
   \nabla m_2 \cdot \nabla (m_2)_z  = 0 ,  \, \, \, \nabla m_3 \cdot \nabla (m_3)_z  = 0 ,
   \quad \mbox{on} \, \, \Gamma_z .  \label{scheme-BC-4}
\end{equation}
Meanwhile, on the right hand side of~\eqref{scheme-BC-2}, we see that the first and third terms, as well as the first two parts in the second and fourth terms, disappear, which comes from the homogeneous Neumann boundary condition for $m_2$ and $m_3$. Then we arrive at 
\begin{equation} 
  \alpha  \partial_z^3 m_1  =  - m_3  \partial_z^3 m_2  + m_2  \partial_z^3 m_3 , 
  \quad \mbox{on} \, \, \Gamma_z .  \label{scheme-BC-5-1}
\end{equation} 
Similarly, we are able to derive the following equalities; 
\begin{equation} 
\begin{aligned} 
  & 
  \alpha  \partial_z^3 m_2  =  m_1  \partial_z^3 m_3  - m_3  \partial_z^3 m_1 ,  
  \quad \mbox{on} \, \, \Gamma_z  
\\
 & 
  \alpha  \partial_z^3 m_3  =  m_2  \partial_z^3 m_1  - m_1  \partial_z^3 m_2 , 
  \quad \mbox{on} \, \, \Gamma_z .  
\end{aligned} 
\label{scheme-BC-5-2}
\end{equation} 
In turn, for any $\alpha >0$, we observe that the matrix $\left( \begin{array}{ccc} 
 \alpha & m_3    & - m_2 \\
 - m_3  & \alpha  & m_1  \\ 
 m_2     & - m_1  & \alpha 
\end{array} \right)$ has a positive determinant, so that the linear system~\eqref{scheme-BC-5-1}-\eqref{scheme-BC-5-2} has only trivial solution: 
\begin{equation} 
  \partial_z^3 m_1  =  \partial_z^3 m_2 = \partial_z^3 m_3 = 0  , 
  \quad \mbox{on} \, \, \Gamma_z .  \label{scheme-BC-5-3}
\end{equation} 
As a result, a substitution into~\eqref{exact-4} yields an ${\mathcal O} (h^5)$ consistency accuracy for the symmetric extrapolation: 
\begin{equation} 
\begin{aligned} 
        & 
	\m_e  (\hat{x}_i, \hat{y}_j, \hat{z}_0 )
	= \m_e (\hat{x}_i, \hat{y}_j, \hat{z}_1 )
	+   \mathcal{O}(h^5) , 
\\
  & 
	\underline{\m}  (\hat{x}_i, \hat{y}_j, \hat{z}_0 )
	= \underline{\m} (\hat{x}_i, \hat{y}_j, \hat{z}_1 )
	+   \mathcal{O}(h^5) .  
\end{aligned} 
   \label{exact-5}
\end{equation}
In other words, the extrapolation formula~\cref{exact-3} is indeed $\mathcal{O}(h^5)$ accurate.
	
Subsequently, a detailed calculation of Taylor expansion, in both time and space, leads to the following truncation error estimate:
\begin{equation} 
	\begin{aligned} 
	  & 
	\frac{\frac{11}{6} \underline{\m}_h^{n+3} - 3 \underline{\m}_h^{n+2}
		+ \frac32 \underline{\m}_h^{n+1} - \frac13 \underline{\m}_h^n }{k} 
\\
		= & - \hat{\underline{\m}}_h^{n+3} 
		\times \Delta_{h, (4)} \underline{\m}_h^{n+3} + \tau^{n+3} 
	  + \alpha \Delta_{h, (4)} \underline{\m}_h^{n+3} 
	+ \alpha | \tilde{\nabla}_{h, (4)} \hat{\underline{\m}}_h^{n+3} |^2 
	\hat{\underline{\m}}_h^{n+3} , 
	\end{aligned} 
	\label{consistency-2}	 
\end{equation} 	
with $\hat{\underline{\m}}_h^{n+3} = 3 \underline{\m}_h^{n+2} - 3 \underline{\m}_h^{n+1} + \underline{\m}_h^n$, and $\| \tau^{n+3} \|_2 \le \mathcal{C} (k^3+h^4)$. In addition, a higher order Taylor expansion in space and time reveals the following estimate for the discrete gradient of the truncation error:  
\begin{equation} 
  \| \nabla _h \tau^{n+1} \|_2 \le {\mathcal C} (k^3 + h^4) .   \label{truncation error-1}
\end{equation}
In fact, such a discrete $\| \cdot \|_{H_h^1}$ bound for the truncation comes from the regularity assumption for the exact solution, $\m_e \in C^4 ([0,T]; C^0) \cap C^3 ([0, T]; C^1) \cap L^{\infty}([0,T]; C^6)$, as stated in Theorem~\ref{cccthm2}. 	
	
In turn, we introduce the numerical error functions $\tilde{\e}_h^n=\underline{\m}_h^n-\tilde{\m}_h^n$, ${\e}_h^n=\underline{\m}_h^n-\m_h^n$, at a point-wise level. A subtraction of \cref{scheme-1-1}-\cref{scheme-1-2} from the consistency estimate~\cref{consistency-2} leads to
	the error function evolution system:
\begin{equation} 
	\begin{aligned} 
	  & 
	\frac{\frac{11}{6} \tilde{\e}_h^{n+3} - 3 \tilde{\e}_h^{n+2} + \frac32 \tilde{\e}_h^{n+1} 
	- \frac13 \tilde{\e}_h^n }{k}  
\\
   = & 
	- \hat{\m}_h^{n+3} \times \Delta_{h, (4)} \hat{\tilde{\e}}_h^{n+3} 
	- \hat{\e}_h^{n+3} \times \Delta_{h, (4)} \hat{\um}_h^{n+3} 
	+ \tau^{n+3}  	
	+ \alpha \Delta_{h, (4)} \tilde{\e}_h^{n+3}  
\\
  & 	
	+ \alpha |  \tilde{\nabla}_{h, (4)} \hat{\underline{\m}}_h^{n+3} |^2  
	\hat{\e}_h^{n+3} 
	+ \alpha \Big( 
	\tilde{\nabla}_{h, (4)}  ( \hat{\underline{\m}}_h^{n+3} + \hat{\m}_h^{n+3} ) 
	\cdot  \tilde{\nabla}_{h, (4)}  \hat{\e}_h^{n+3}  \Big) 		 
	\hat{\m}_h^{n+3}  , 
	\end{aligned} 
	\label{ccc73} 
\end{equation} 	
with $\hat{\e}_h^{n+3} = 3 \e_h^{n+2} - 3 \e_h^{n+1} + \e_h^n$, $\hat{\tilde{\e}}_h^{n+3} = 3 \tilde{\e}_h^{n+2} - 3 \tilde{\e}_h^{n+1} + \tilde{\e}_h^n$. 
	
	Before we proceed into the formal error estimate, we establish the bound for the exact solution ${\um}$ and the numerical solution $\m_h$. For the profile $\um \in L^{\infty}([0,T]; C^6)$, which turns out to be the exact solution, we still use $\mathcal{C}$ to denote its bound:
	\begin{align}
	\|\nabla_h^r \um_h \|_{\infty}\leq \mathcal{C}, \quad r=0,1,2,3 ,  \label{bound-1}
	\end{align}
in which $\um_h = {\mathcal P}_h \um$, the point-wise interpolation of the continuous function $\um$. In addition, we make the following \textit{a priori} assumption for the numerical error function:
	\begin{equation} \label{bound-2}
	\|{\e}_h^k \|_{\infty} \le \dt + h , \, \, \,  \|\nabla_h{\e}_h^k \|_{\infty} \le \frac13 ,  \, \, \,
	\| \tilde{\e}_h^k \|_{\infty}+\|\nabla_h \tilde{\e}_h^k \|_{\infty} \le \frac13 ,
	\, \, \,  \mbox{$\ell \le k \le \ell+2$} .
	\end{equation}
	Such an assumption will be recovered by the convergence analysis at time step $t^{\ell+3}$. In turn, an application of triangle inequality yields the desired $W_h^{1,\infty}$ bound for the numerical solutions $\m_h$ and $\tilde{\m}_h$:
	\begin{align}
	\|{\m}_h^k \|_{\infty} &= \|\um_h^k -{\e}_h^k \|_{\infty}\leq\|\um_h^k \|_{\infty}+\|{\e}_h^k \|_{\infty}\leq \mathcal{C}+ \frac13 ,  \label{bound-3} \\
	\|\nabla_h{\m}_h^k \|_{\infty} &= \|\nabla_h\um_h^k -\nabla_h{\e}_h^k \|_{\infty}\leq\|\nabla_h\um_h^k \|_{\infty}+\|\nabla_h{\e}_h^k \|_{\infty}\leq \mathcal{C}+\frac13 , \nonumber \\
	\| \tilde{\m}_h^k \|_{\infty} & \leq \mathcal{C}+ \frac13 , \quad
	\| \nabla_h \tilde{\m}_h^k \|_{\infty} \leq \mathcal{C}+ \frac13
	\quad \mbox{(similar derivation)} . \label{bound-4}  			
	\end{align}
Furthermore, we need a sharper $\| \cdot \|_\infty$ bound for $\hat{\m}_h^{\ell+3} = 3 \m_h^{\ell+2} - 3 \m_h^{\ell +1} + \m_h^\ell$, which will be needed in the later error analysis. The following extrapolation estimate is valid, due to the $C^4 (0,T; C^0)$ regularity of the exact solution $\m_e$: 
\begin{equation} 
  \um_h^{\ell+3} = 3 \um_h^{\ell+2} - 3 \um_h^{\ell+1} + \um_\ell + {\mathcal O} (\dt^3) .  
  \label{bound-5-1} 
\end{equation} 
Meanwhile, since $| \m_e | \equiv 1$, we conclude that 
\begin{equation} 
  \| 3 \um_h^{\ell+2} - 3 \um_h^{\ell+1} + \um_\ell \|_\infty  \le 1 + {\mathcal C} \dt^3 .  \label{bound-5-2} 
\end{equation}  
As a result, its combination with the a-priori assumption that $\|{\e}_h^k \|_{\infty} \le \dt + h$, for $k=\ell, \ell+1, \ell+2$, (as given by~\eqref{bound-2}), implies that 
\begin{equation} 
\begin{aligned} 
  \| \hat{\m}_h^{\ell+3} \|_\infty = & \| 3 \m_h^{\ell+2} - 3 \m_h^{\ell+1}  + \m_h^\ell \|_\infty   
\\
  \le & 
     \| 3 \um_h^{\ell+2} - 3  \um_h^{\ell +1} + \um_h^\ell \|_\infty  
  + \| 3 \e_h^{\ell+2} - 3 \e_h^{\ell +1} + \e_h^\ell \|_\infty    
\\
  \le & 1 + {\mathcal C} \dt^3 + 7 (\dt + h)  \le \beta_1 := 1 + \frac{\alpha - 7}{14}  , 
\end{aligned} 
\label{bound-5-3}  
\end{equation}   
provided that $\dt$ and $h$ are sufficiently small. In addition, we denote $\gamma_0 := \alpha- 7 >0$, so that $\beta_1 = 1 + \frac{1}{14} \gamma_0$. 
	
	Then we perform a discrete $L^2$ error estimate at $t^{\ell+3}$ using the mathematical induction. 
	Motivated by the telescope formula~\eqref{BDF-3-est-0} (for the BDF3 temporal stencil)  in Lemma~\ref{lem:3rd order BDF}, we take a discrete inner product with the numerical error equation \cref{ccc73} by $2 \tilde{\e}_h^{\ell+3} - \tilde{\e}_h^{\ell+2}$ and obtain 
	\begin{align}\label{rhs}
	R.H.S.&= \left\langle-\left( 3 {\m}_h^{\ell+2} - 3 \m_h^{\ell+1} + \m_h^\ell \right) 
	\times \Delta_{h, (4)} \hat{\tilde{\e}}_h^{\ell+3}, 2 \tilde{\e}_h^{\ell+3} - \tilde{\e}_h^{\ell+2} \right\rangle\\
	&-\left\langle \hat{{\e}}_h^{\ell+3} \times \Delta_{h, (4)} \hat{\um}_h^{\ell+3}, 
	2 \tilde{\e}_h^{\ell+3} - \tilde{\e}_h^{\ell+2} \right\rangle 
	+ \left\langle \tau^{\ell+3}, 2 \tilde{\e}_h^{\ell+3} - \tilde{\e}_h^{\ell+2} \right\rangle 
	\nonumber\\
	& - \alpha \langle \nabla_{h, (4)} \tilde{\e}_h^{\ell+3} , 
	\nabla_{h, (4)} ( 2 \tilde{\e}_h^{\ell+3} - \tilde{\e}_h^{\ell+2} ) \rangle \nonumber \\ 
	& 
	+ \alpha \langle | \tilde{\nabla}_{h, (4)} \hat{\underline{\m}}_h^{\ell+3} |^2 
	\hat{\e}_h^{\ell+3} , 2 \tilde{\e}_h^{\ell+3} - \tilde{\e}_h^{\ell+2} \rangle  \nonumber\\
	& + \alpha \Big\langle \Big( 
	\tilde{\nabla}_{h, (4)}  ( \hat{\underline{\m}}_h^{\ell+3} + \hat{\m}_h^{\ell+3} ) 
	\cdot \tilde{\nabla}_{h, (4)} \hat{\e}_h^{\ell+3}  \Big) 		 
	\hat{\m}_h^{\ell+3} ,  2 \tilde{\e}_h^{\ell+3} - \tilde{\e}_h^{\ell+2} \Big\rangle  \nonumber  \\
	&=: \tilde{I}_1+\tilde{I}_2+\tilde{I}_3 + \tilde{I}_4   
	+\tilde{I}_5+\tilde{I}_6.\nonumber
	\end{align}
	
	\begin{itemize}
		\item Estimate of $\tilde{I}_1$: A combination of the summation by parts formula~\cref{sum1} (notice that the numerical error function $\tilde{\e}$ satisfies the homogeneous Neumann boundary condition~\cref{BC-1}) and inequality~\cref{lem27_1} results in 	
		\begin{align}\label{I1}
		\tilde{I}_1 = \, & \left\langle-\left( 3 {\m}_h^{\ell+2} - 3 \m_h^{\ell+1} + \m_h^\ell \right) 
	\times \Delta_{h, (4)} \hat{\tilde{\e}}_h^{\ell+3}, 
	2 \tilde{\e}_h^{\ell+3} - \tilde{\e}_h^{\ell+2} \right\rangle \\
		= \, & \left\langle ( 2 \tilde{\e}_h^{\ell+3} - \tilde{\e}_h^{\ell+2} ) 
		\times \left( 3 {\m}_h^{\ell+2} - 3 \m_h^{\ell+1} + \m_h^\ell \right), 
		-\Delta_{h, (4)} \hat{\tilde{\e}}_h^{\ell+3} \right\rangle \nonumber\\
		= \, & \left\langle \nabla_{h, (4)} \Big[  
		( 2 \tilde{\e}_h^{\ell+3} - \tilde{\e}_h^{\ell+2} ) 
		\times \left( 3 \m_h^{\ell+2} - 3 \m_h^{\ell+1} + \m_h^\ell \right)	 \Big] , 
		\nabla_{h, (4)} \hat{\tilde{\e}}_h^{\ell+3}  \right\rangle \nonumber\\
		\le \, & \mathcal{C}\Big(\|\nabla_{h, (4)} \hat{\tilde{\e}}_h^{\ell+3}\|_2^2 
		+ \|\nabla_h ( 2 \tilde{\e}_h^{\ell+3} - \tilde{\e}_h^{\ell+2} ) \|_2^2 
		\cdot \| 3 \m_h^{\ell+2} - 3 \m_h^{\ell+1} + \m_h^\ell \|_{\infty}^2 \nonumber \\
		& +\| 2 \tilde{\e}_h^{\ell+3} - \tilde{\e}_h^{\ell+2} \|_2^2 
		\cdot \|\nabla_h ( 3 \m_h^{\ell+2} - 3 \m_h^{\ell+1} 
		+ \m_h^\ell )\|_{\infty}^2 \Big) \nonumber\\
		\le \, & \mathcal{C} ( \|\nabla_h\tilde{\e}_h^{\ell+3}\|_2^2  
		+ \|\nabla_h\tilde{\e}_h^{\ell+2} \|_2^2	
		+ \| \nabla_h\tilde{\e}_h^{\ell+1} \|_2^2  
		+ \| \nabla_h\tilde{\e}_h^{\ell} \|_2^2	 \nonumber \\
		& \quad 	
		+ \| \tilde{\e}_h^{\ell+3} \|_2^2 			
		+ \| \tilde{\e}_h^{\ell+2} \|_2^2	 ) , \nonumber
		\end{align}
in which the priori-bound~\eqref{bound-3} and the preliminary estimate~\eqref{gradient-L2-0-1} have also been applied. 
		
		\item Estimate of $\tilde{I}_2$:
		\begin{align}\label{I2}
		\tilde{I}_2 =\, & -\left\langle \hat{{\e}}_h^{\ell+3} \times \Delta_{h, (4)} \hat{\um}_h^{\ell+3}, 
	2 \tilde{\e}_h^{\ell+3} - \tilde{\e}_h^{\ell+2} \right\rangle  \\
		\le \, & \frac{1}{2}\big[ \| 2 \tilde{\e}_h^{\ell+3} - \tilde{\e}_h^{\ell+2} \|_2^2 
		+ \| 3 \e_h^{\ell+2} - 3 \e_h^{\ell +1} + \e_h^\ell \|_2^2 
		\cdot \| \Delta_{h, (4)} \hat{\um}_h^{\ell+3 }\|_{\infty}^2 \big] \nonumber\\
		\le \, & \mathcal{C} ( \|\tilde{\e}_h^{\ell+3}\|_2^2 + \|\tilde{\e}_h^{\ell+2}\|_2^2 
		+ \| \e_h^{\ell+2}\|_2^2 + \| \e_h^{\ell+1}\|_2^2
		+ \| \e_h^{\ell}\|_2^2 ) , \nonumber
		\end{align}
in which the bound for $\|\Delta_{h, (4)} \hat{\um}_h^{\ell+3}\|_{\infty}$ is given by the preliminary estimate~\eqref{bound-1}, with $r=2$.

		\item Estimate of the truncation error term $\tilde{I}_3$: An application of Cauchy inequality gives
		\begin{align}\label{I3}
		\tilde{I}_3 = \left\langle \tau^{\ell+3} , 
		2 \tilde{\e}_h^{\ell+3} - \tilde{\e}_h^{\ell+2} \right\rangle
		 \leq \mathcal{C} ( \| \tilde{\e}_h^{\ell+3} \|_2^2 
		 + \| \tilde{\e}_h^{\ell+2} \|_2^2 ) + \mathcal{C} (k^6 + h^8) .
		\end{align}
		
		\item Estimate of $\tilde{I}_4$: The following equalities are available: 
		\begin{align} \label{I4-1}  
		  & 
		\langle \nabla_{h, (4)} \tilde{\e}_h^{\ell+3} , 
	\nabla_{h, (4)} ( \tilde{\e}_h^{\ell+3} - \tilde{\e}_h^{\ell+2} ) \rangle  \\ 
	         = \, & \frac12 ( \| \nabla_{h, (4)} \tilde{\e}_h^{\ell+3} \|_2^2  
	          - \| \nabla_{h, (4)} \tilde{\e}_h^{\ell+2} \|_2^2   
	        + \| \nabla_{h, (4)} ( \tilde{\e}_h^{\ell+3} - \tilde{\e}_h^{\ell+2} ) \|_2^2  )  ,   \nonumber 
		\end{align}		
		
		\begin{align} \label{I4}
		\tilde{I}_4 = \, &  - \alpha \langle \nabla_{h, (4)} \tilde{\e}_h^{\ell+3} , 
	\nabla_{h, (4)} ( 2 \tilde{\e}_h^{\ell+3} - \tilde{\e}_h^{\ell+2} ) \rangle  \\  
	       = \, &  - \alpha \| \nabla_{h, (4)} \tilde{\e}_h^{\ell+3} \|_2^2 
	       - \alpha  \langle \nabla_{h, (4)} \tilde{\e}_h^{\ell+3} , 
	\nabla_{h, (4)} ( \tilde{\e}_h^{\ell+3} - \tilde{\e}_h^{\ell+2} ) \rangle  \nonumber \\
	       = \, & - \alpha \| \nabla_{h, (4)} \tilde{\e}_h^{\ell+3} \|_2^2  
	        - \frac{\alpha}{2} ( \| \nabla_{h, (4)} \tilde{\e}_h^{\ell+3} \|_2^2  
	          - \| \nabla_{h, (4)} \tilde{\e}_h^{\ell+2} \|_2^2  )   \nonumber \\
	          \, & 
	         - \frac{\alpha}{2} \| \nabla_{h, (4)} ( \tilde{\e}_h^{\ell+3} - \tilde{\e}_h^{\ell+2} ) \|_2^2 .  	        \nonumber  	       
		\end{align}
				
		\item Estimate of $\tilde{I}_5$: An application of discrete H\"older inequality gives 
          	\begin{align}\label{I5-1}
 		\Big\| | \tilde{\nabla}_{h, (4)} \hat{\underline{\m}}_h^{\ell+3} |^2 
 	\hat{\e}_h^{\ell+3} \Big\|_2 
 		\le \, & \| \tilde{\nabla}_{h, (4)} \hat{\underline{\m}}_h^{\ell+3} \|_\infty^2 
 		\cdot  \| \hat{\e}_h^{\ell+3} \|_2     \\
 		\le \, &  {\mathcal C}  \| \hat{\e}_h^{\ell+3} \|_2  , \nonumber 
         	\end{align}
in which the $W_h^{1, \infty}$ bound~\eqref{bound-1} for the exact solution has been applied. 
				
 	\begin{align}\label{I5}
 		\tilde{I}_5 = \,& \alpha \langle | \tilde{\nabla}_{h, (4)} \hat{\underline{\m}}_h^{\ell+3} |^2 
	\hat{\e}_h^{\ell+3} , 2 \tilde{\e}_h^{\ell+3} - \tilde{\e}_h^{\ell+2} \rangle  \\
		\le \, & \alpha \Big\| | \tilde{\nabla}_{h, (4)} \hat{\underline{\m}}_h^{\ell+3} |^2 
 	\hat{\e}_h^{\ell+3} \Big\|_2  
	\cdot \| 2 \tilde{\e}_h^{\ell+3} - \tilde{\e}_h^{\ell+2} \|_2  \nonumber \\
		\le \, &  {\mathcal C} \alpha  \| \hat{\e}_h^{\ell+3} \|_2  
		\cdot \|  2 \tilde{\e}_h^{\ell+3} - \tilde{\e}_h^{\ell+2}  \|_2 \nonumber\\ 
		\le \, & {\mathcal C} (  \| \e_h^{\ell+2}\|_2^2 + \| \e_h^{\ell+1}\|_2^2 
		+ \| \e_h^{\ell} \|_2^2	+ \|\tilde{\e}_h^{\ell+3} \|_2^2 			
		+ \|\tilde{\e}_h^{\ell+2}\|_2^2 ) . \nonumber
 	\end{align}
		
		\item Estimate $\tilde{I}_6$: Similarly, an application of discrete H\"older inequality gives 
          	\begin{align}\label{I6-1}
	        & 
 		\Big\| \Big( 
	 \tilde{\nabla}_{h, (4)}  ( \hat{\underline{\m}}_h^{\ell+3} + \hat{\m}_h^{\ell+3} ) 
	\cdot  \tilde{\nabla}_{h, (4)}  \hat{\e}_h^{\ell+3}  \Big) 		 
	\hat{\m}_h^{\ell+3} \Big\|_2  \\ 
 		\le \, & \Big(  \| \tilde{\nabla}_{h, (4)}  \hat{\underline{\m}}_h^{\ell+3} \|_\infty) 
		+ \| \tilde{\nabla}_{h, (4)} \hat{\m}_h^{\ell+3} \|_\infty \Big) \cdot  
	 \| \tilde{\nabla}_{h, (4)}  \hat{\e}_h^{\ell+3} \|_2  \cdot \| \hat{\m}_h^{\ell+3} \|_\infty  
	  \nonumber \\ 
	  \le \, & {\mathcal C}  \| \nabla_{h, (4)}  \hat{\e}_h^{\ell+3} \|_2 
	   \le {\mathcal C}  \| \nabla_h \hat{\e}_h^{\ell+3} \|_2	 , \nonumber
         	\end{align}
in which the $W_h^{1, \infty}$ bounds~\eqref{bound-1}, \eqref{bound-3} (for the exact and numerical solutions), as well as the preliminary estimate~\eqref{gradient-L2-0-1}, have been applied.

		\begin{align} \label{I6}
		\tilde{I}_6 = \,& \alpha \Big\langle \Big( 
	\tilde{\nabla}_{h, (4)}  ( \hat{\underline{\m}}_h^{\ell+3} + \hat{\m}_h^{\ell+3} ) 
	\cdot \tilde{\nabla}_{h, (4)} \hat{\e}_h^{\ell+3}  \Big) 		 
	\hat{\m}_h^{\ell+3} ,  2 \tilde{\e}_h^{\ell+3} - \tilde{\e}_h^{\ell+2} \Big\rangle  \\
		\le \, & \alpha \Big\| \Big( 
	 \tilde{\nabla}_{h, (4)}  ( \hat{\underline{\m}}_h^{\ell+3} + \hat{\m}_h^{\ell+3} ) 
	\cdot  \tilde{\nabla}_{h, (4)}  \hat{\e}_h^{\ell+3}  \Big) 		 
	\hat{\m}_h^{\ell+3} \Big\|_2 
	  \cdot  \| 2 \tilde{\e}_h^{\ell+3} - \tilde{\e}_h^{\ell+2} \|_2  \nonumber\\ 
	        \le \, & {\mathcal C} \alpha \| \nabla_h  \hat{\e}_h^{\ell+3} \|_2 
	        \cdot 	 \| 2 \tilde{\e}_h^{\ell+3} - \tilde{\e}_h^{\ell+2}  \|_2   \nonumber \\ 	               
		\le \, & \mathcal{C} ( \|\tilde{\e}_h^{\ell+3} \|_2^2 + \|\tilde{\e}_h^{\ell+2} \|_2^2 
		+ \| \nabla_h \e_h^{\ell+2} \|_2^2 + \| \nabla_h \e_h^{\ell+1} \|_2^2
		+ \| \nabla_h \e_h^\ell \|_2^2 ) . \nonumber
		\end{align}

	\end{itemize}	
	Meanwhile, by the telescope formula~\eqref{BDF-3-est-0}, the inner product of the left hand side of \cref{ccc73} with $2 \tilde{\e}_h^{\ell+3} - \tilde{\e}_h^{\ell+2}$ turns out to be	
	\begin{align*}
	L.H.S. = &  \frac{1}{k} \Big( 
	\| \alpha_1 \tilde{\e}_h^{\ell+3} \|_2^2 - \| \alpha_1 \tilde{\e}_h^{\ell+2} \|_2^2
 + \| \alpha_2 \tilde{\e}_h^{\ell+3} + \alpha_3 \tilde{\e}_h^{\ell+2} \|_2^2
 - \| \alpha_2 \tilde{\e}_h^{\ell+2} + \alpha_3 \tilde{\e}_h^{\ell+1} \|_2 ^2 \nonumber
\\
  & \quad 
  + \| \alpha_4 \tilde{\e}_h^{\ell+3} + \alpha_5 \tilde{\e}_h^{\ell+2} 
   + \alpha_6 \tilde{\e}_h^{\ell+1} \|_2^2
 - \| \alpha_4 \tilde{\e}_h^{\ell+2} + \alpha_5 \tilde{\e}_h^{\ell+1} 
 + \alpha_6 \tilde{\e}_h^{\ell} \|_2^2 \nonumber
\\
  & \quad 
    + \| \alpha_7 \tilde{\e}_h^{\ell+3} + \alpha_8 \tilde{\e}_h^{\ell+2} 
    + \alpha_9 \tilde{\e}_h^{\ell+1}
   + \alpha_{10} \tilde{\e}_h^{\ell} \|_2^2  \Big) .\nonumber
	\end{align*}
	Its combination with \cref{I1,I2,I3,I4,I5,I6} and \cref{rhs} leads to
	\begin{align}\label{ccc34}
	& 	\| \alpha_1 \tilde{\e}_h^{\ell+3} \|_2^2 - \| \alpha_1 \tilde{\e}_h^{\ell+2} \|_2^2
 + \| \alpha_2 \tilde{\e}_h^{\ell+3} + \alpha_3 \tilde{\e}_h^{\ell+2} \|_2^2
 - \| \alpha_2 \tilde{\e}_h^{\ell+2} + \alpha_3 \tilde{\e}_h^{\ell+1} \|_2 ^2 
\\
  & 
  + \| \alpha_4 \tilde{\e}_h^{\ell+3} + \alpha_5 \tilde{\e}_h^{\ell+2} 
   + \alpha_6 \tilde{\e}_h^{\ell+1} \|_2^2
 - \| \alpha_4 \tilde{\e}_h^{\ell+2} + \alpha_5 \tilde{\e}_h^{\ell+1} 
 + \alpha_6 \tilde{\e}_h^{\ell} \|_2^2  \nonumber 
\\
  & 
   + \frac{\alpha}{2} k ( \| \nabla_{h, (4)} \tilde{\e}_h^{\ell+3} \|_2^2  
	          - \| \nabla_{h, (4)} \tilde{\e}_h^{\ell+2} \|_2^2  )   \nonumber \\ 
	\le \, & \mathcal{C} k \Big( \|\nabla_h\tilde{\e}_h^{\ell+3}\|_2^2 
	+ \|\nabla_h\tilde{\e}_h^{\ell+2}\|_2^2  
	+ \|\nabla_h\tilde{\e}_h^{\ell+1}\|_2^2 
	+ \|\nabla_h\tilde{\e}_h^{\ell}\|_2^2 	\nonumber \\
	& \quad 
	+ \| \nabla_h \e_h^{\ell+2} \|_2^2 		
	+ \| \nabla_h \e_h^{\ell+1} \|_2^2 + \| \nabla_h \e_h^{\ell}\|_2^2 \nonumber \\ 
        & \quad 	
	+ \|\tilde{\e}_h^{\ell+3} \|_2^2 + \|\tilde{\e}_h^{\ell+2} \|_2^2 
	+ \| \e_h^{\ell+2} \|_2^2 + \| \e_h^{\ell+1}\|_2^2 + \| \e_h^{\ell}\|_2^2 \Big) 
	+ \mathcal{C} k (k^6 + h^8) 	. \nonumber
	\end{align}
	
	However, the standard $L^2$ error estimate~\cref{ccc34} does not allow one to apply discrete Gronwall inequality, due to the $H_h^1$ norms of the error function involved on the right hand side. To overcome this difficulty, we take a discrete inner product with the numerical error equation~\cref{ccc73} by $-\Delta_{h, (4)} ( 2 \tilde{\e}_h^{\ell+3} - \tilde{\e}_h^{\ell+2} )$ and see that
	\begin{align}
		R.H.S.&= \left\langle - \hat{\m}_h^{\ell+3}   
	\times \Delta_{h, (4)} \hat{\tilde{\e}}_h^{\ell+3}, 
	-\Delta_{h, (4)} ( 2 \tilde{\e}_h^{\ell+3} - \tilde{\e}_h^{\ell+2} ) \right\rangle \\
	&-\left\langle \hat{{\e}}_h^{\ell+3} \times \Delta_{h, (4)} \hat{\um}_h^{\ell+3}, 
	-\Delta_{h, (4)} ( 2 \tilde{\e}_h^{\ell+3} - \tilde{\e}_h^{\ell+2} ) \right\rangle \nonumber \\ 
	& + \left\langle \tau^{\ell+3}, 
	-\Delta_{h, (4)} ( 2 \tilde{\e}_h^{\ell+3} - \tilde{\e}_h^{\ell+2} )  \right\rangle 
	\nonumber\\
	& - \alpha \langle \Delta_{h, (4)} \tilde{\e}_h^{\ell+3} , 
	\Delta_{h, (4)} ( 2 \tilde{\e}_h^{\ell+3} - \tilde{\e}_h^{\ell+2} ) \rangle \nonumber \\ 
	& 
	+ \alpha \langle | \tilde{\nabla}_{h, (4)} \hat{\underline{\m}}_h^{\ell+3} |^2 
	\hat{\e}_h^{\ell+3} , 
	-\Delta_{h, (4)} ( 2 \tilde{\e}_h^{\ell+3} - \tilde{\e}_h^{\ell+2} )  \rangle  \nonumber\\
	& + \alpha \Big\langle \Big( 
	\tilde{\nabla}_{h, (4)}  ( \hat{\underline{\m}}_h^{\ell+3} + \hat{\m}_h^{\ell+3} ) 
	\cdot \tilde{\nabla}_{h, (4)} \hat{\e}_h^{\ell+3}  \Big) 		 
	\hat{\m}_h^{\ell+3} ,  
	-\Delta_{h, (4)} ( 2 \tilde{\e}_h^{\ell+3} - \tilde{\e}_h^{\ell+2} )  \Big\rangle  \nonumber  \\
	&=: I_1+ I_2+ I_3 + I_4  + I_5 + I_6 . \nonumber
	\end{align}
	\begin{itemize}
		\item Estimate of $I_1$: 
		 
		\begin{align} \label{I1-2-1}		
		  & 
		\| \Delta_{h, (4)} \hat{\tilde{\e}}_h^{\ell+3} \|_2 
		\cdot \| \Delta_{h, (4)} \tilde{\e}_h^{\ell+3}  \|_2 	 \\
		\le \, &  ( 3 \| \Delta_{h, (4)} \tilde{\e}_h^{\ell+2} \|_2 
		 + 2 \| \Delta_{h, (4)} \tilde{\e}_h^{\ell+1} \|_2  
		 +   \| \Delta_{h, (4)} ( \tilde{\e}_h^{\ell+1} - \tilde{\e}_h^{\ell} ) \|_2 	) 	 		
		\cdot \| \Delta_{h, (4)} \tilde{\e}_h^{\ell+3}  \|_2  \nonumber \\ 
		\le \, & \frac32 \| \Delta_{h, (4)} \tilde{\e}_h^{\ell+2} \|_2^2 
		 + \| \Delta_{h, (4)} \tilde{\e}_h^{\ell+1} \|_2^2   
		 +   \frac14 \| \Delta_{h, (4)} 
		 ( \tilde{\e}_h^{\ell+1} - \tilde{\e}_h^{\ell} ) \|_2^2  	 		
		+ \frac72 \| \Delta_{h, (4)} \tilde{\e}_h^{\ell+3}  \|_2^2 , \nonumber  
		\end{align} 
		
		\begin{align} \label{I1-2-2}			
		  & 
		\| \Delta_{h, (4)} \hat{\tilde{\e}}_h^{\ell+3} \|_2 
		\cdot \| \Delta_{h, (4)}  ( \tilde{\e}_h^{\ell+3} - \tilde{\e}_h^{\ell+2} )  \|_2  \\
		\le \, &  ( \| \Delta_{h, (4)} \tilde{\e}_h^{\ell+2} \|_2 
		 + 2 \| \Delta_{h, (4)} ( \tilde{\e}_h^{\ell+2} - \tilde{\e}_h^{\ell+1} ) \|_2  
		 +   \| \Delta_{h, (4)} ( \tilde{\e}_h^{\ell+1} - \tilde{\e}_h^{\ell} ) \|_2 	)  
		 \nonumber \\ 
		 & \quad 	 		
		\cdot \| \Delta_{h, (4)}  ( \tilde{\e}_h^{\ell+3} - \tilde{\e}_h^{\ell+2} )  \|_2 
		   \nonumber \\ 
		\le \, & \| \Delta_{h, (4)} \tilde{\e}_h^{\ell+2} \|_2^2  
		 + \| \Delta_{h, (4)} ( \tilde{\e}_h^{\ell+2} - \tilde{\e}_h^{\ell+1} ) \|_2^2   
		 +   \frac12 \| \Delta_{h, (4)} ( \tilde{\e}_h^{\ell+1} - \tilde{\e}_h^{\ell} ) \|_2^2   
		 \nonumber \\ 
		 & \quad 	 		
		 + \frac74 \| \Delta_{h, (4)}  ( \tilde{\e}_h^{\ell+3} - \tilde{\e}_h^{\ell+2} )  \|_2^2   ,  
		 \nonumber  		
		\end{align}

		\begin{align} \label{ccc11}		
		I_1 = \, & \left\langle - \hat{\m}_h^{\ell+3}   
	\times \Delta_{h, (4)} \hat{\tilde{\e}}_h^{\ell+3}, 
	-\Delta_{h, (4)} ( 2 \tilde{\e}_h^{\ell+3} - \tilde{\e}_h^{\ell+2} ) \right\rangle  \\ 	
		\le \, & \| 3 \m_h^{\ell+2} - 3 \m_h^{\ell+1} + \m_h^\ell  \|_\infty 
		\cdot \| \Delta_{h, (4)} \hat{\tilde{\e}}_h^{\ell+3} \|_2 
		\cdot \| \Delta_{h, (4)} ( 2 \tilde{\e}_h^{\ell+3} - \tilde{\e}_h^{\ell+2} )  \|_2 \nonumber  \\ 
		\le \, & \beta_1 ( \| \Delta_{h, (4)} \hat{\tilde{\e}}_h^{\ell+3} \|_2 
		\cdot \| \Delta_{h, (4)} \tilde{\e}_h^{\ell+3}  \|_2	 \nonumber \\ 
		& \quad 	
		+ \| \Delta_{h, (4)} \hat{\tilde{\e}}_h^{\ell+3} \|_2 
		\cdot \| \Delta_{h, (4)} ( \tilde{\e}_h^{\ell+3} - \tilde{\e}_h^{\ell+2} )  \|_2 )  \nonumber \\ 
		\le \, & \beta_1 \Big( \frac72 \| \Delta_{h, (4)} \tilde{\e}_h^{\ell+3} \|_2^2 
		 +  \frac52 \| \Delta_{h, (4)} \tilde{\e}_h^{\ell+2} \|_2^2   
		 +    \| \Delta_{h, (4)} \tilde{\e}_h^{\ell+1} \|_2^2  	 \nonumber \\	
		 & \quad 	 
		  + \frac74 \| \Delta_{h, (4)} ( \tilde{\e}_h^{\ell+3} - \tilde{\e}_h^{\ell+2} ) \|_2^2   
		 +      \| \Delta_{h, (4)} ( \tilde{\e}_h^{\ell+2} - \tilde{\e}_h^{\ell+1} ) \|_2^2  
		 \nonumber \\
		 & \quad 
		 + \frac34 \| \Delta_{h, (4)}  ( \tilde{\e}_h^{\ell+1} - \tilde{\e}_h^{\ell} )  \|_2^2 \Big) , 
		 \nonumber 	
		\end{align}
in which the preliminary $\| \cdot \|_\infty$ bound~\eqref{bound-5-3} has been applied.

		\item Estimate of $I_2$: 		
		\begin{align}\label{ccc14}
		I_2 =\, & -\left\langle \hat{{\e}}_h^{\ell+3} \times \Delta_{h, (4)} \hat{\um}_h^{\ell+3}, 
	-\Delta_{h, (4)} ( 2 \tilde{\e}_h^{\ell+3} - \tilde{\e}_h^{\ell+2} ) \right\rangle  \\
		= \, & \left\langle \nabla_{h, (4)} \left( \Delta_{h, (4)} \hat{\um}_h^{\ell+3} 
		\times ( 3 \e_h^{\ell+2} - 3 \e_h^{\ell+1} + \e_h^\ell ) \right), 
		\nabla_{h, (4)} ( 2 \tilde{\e}_h^{\ell+3} - \tilde{\e}_h^{\ell+2} ) \right\rangle \nonumber\\
		\le \, & \mathcal{C} \Big( 
		\| \nabla_{h, (4)} ( 2 \tilde{\e}_h^{\ell+3} - \tilde{\e}_h^{\ell+2} ) \|_2^2 
		+ \|\Delta_{h, (4)} \hat{\um}_h^{\ell+3} \|_{\infty}^2
		 \cdot\ |\nabla_h ( 3 \e_h^{\ell+2} - 3 \e_h^{\ell+1} + \e_h^\ell ) \|_2^2 \nonumber\\
		&+\|\nabla_h \Delta_{h, (4)} \hat{\um}_h^{\ell+3} \|_{\infty}^2 
		\cdot \| 3 \e_h^{\ell+2} - 3 \e_h^{\ell+1} + \e_h^\ell \|_2^2 \Big) \nonumber \\
		\le \, & \mathcal{C} \Big( \|\nabla_h\tilde{\e}_h^{\ell+3}\|_2^2 
		+ \|\nabla_h\tilde{\e}_h^{\ell+2}\|_2^2		
		+ \|\nabla_h\e_h^{\ell+2} \|_2^2 
		+ \|\nabla_h\e_h^{\ell+1} \|_2^2 + \|\nabla_h\e_h^{\ell}\|_2^2  \nonumber \\
		&  \quad 
		+ \|\e_h^{\ell+2} \|_2^2 
		+ \|\e_h^{\ell+1} \|_2^2+\|\e_h^{\ell}\|_2^2 \Big) .\nonumber
		\end{align}
Similarly, the bound for $\|\nabla_h \Delta_h \hat{\um}_h^{\ell+3}\|_{\infty}$ comes from the preliminary estimate~\eqref{bound-1}, by taking $r=3$.  
		
		\item Estimate of the truncation error term $I_3$:
		\begin{align}  \label{ccc23}		
		I_3 = \, & \left\langle - \Delta_{h, (4)} ( 2 \tilde{\e}_h^{\ell+3} - \tilde{\e}_h^{\ell+2} ) ,
		\tau^{\ell+3}\right\rangle 	 \\ 	
		\le \, & \mathcal{C} ( \|\nabla_{h, (4)} \tilde{\e}_h^{\ell+3} \|_2^2
		  + \|\nabla_{h, (4)} \tilde{\e}_h^{\ell+2} \|_2^2 ) 		
		+\mathcal{C}(k^6 + h^8) \nonumber  \\ 
		\le \, & \mathcal{C} ( \|\nabla_h \tilde{\e}_h^{\ell+3} \|_2^2
		  + \|\nabla_h \tilde{\e}_h^{\ell+2} \|_2^2 ) 		
		+\mathcal{C}(k^6 + h^8) ,  \nonumber 	
		\end{align}		
in which the discrete $H_h^1$ estimate~\eqref{truncation error-1} for the local truncation error has been recalled.  

 		\item Estimate of $I_4$: Similarly, the following equalities are available: 
		\begin{align} \label{I4-1-2}  
		  & 
		\langle \Delta_{h, (4)} \tilde{\e}_h^{\ell+3} , 
	\Delta_{h, (4)} ( \tilde{\e}_h^{\ell+3} - \tilde{\e}_h^{\ell+2} ) \rangle  \\ 
	         = \, & \frac12 ( \| \Delta_{h, (4)} \tilde{\e}_h^{\ell+3} \|_2^2  
	          - \| \Delta_{h, (4)} \tilde{\e}_h^{\ell+2} \|_2^2   
	        + \| \Delta_{h, (4)} ( \tilde{\e}_h^{\ell+3} - \tilde{\e}_h^{\ell+2} ) \|_2^2  )  ,   \nonumber 
		\end{align}		
		
		\begin{align} \label{I4-2}
		I_4 = \, &  - \alpha \langle \Delta_{h, (4)} \tilde{\e}_h^{\ell+3} , 
	\Delta_{h, (4)} ( 2 \tilde{\e}_h^{\ell+3} - \tilde{\e}_h^{\ell+2} ) \rangle  \\  
	       = \, &  - \alpha \| \Delta_{h, (4)} \tilde{\e}_h^{\ell+3} \|_2^2 
	       - \alpha  \langle \Delta_{h, (4)} \tilde{\e}_h^{\ell+3} , 
	\Delta_{h, (4)} ( \tilde{\e}_h^{\ell+3} - \tilde{\e}_h^{\ell+2} ) \rangle  \nonumber \\
	       = \, & - \alpha \| \Delta_{h, (4)} \tilde{\e}_h^{\ell+3} \|_2^2  
	        - \frac{\alpha}{2} ( \| \Delta_{h, (4)} \tilde{\e}_h^{\ell+3} \|_2^2  
	          - \| \Delta_{h, (4)} \tilde{\e}_h^{\ell+2} \|_2^2  )   \nonumber \\
	          \, & 
	         - \frac{\alpha}{2} \| \Delta_{h, (4)} ( \tilde{\e}_h^{\ell+3} - \tilde{\e}_h^{\ell+2} ) \|_2^2 .  	        \nonumber  	       
		\end{align}
	
		\item Estimate of $I_5$: A substitution of the preliminary estimate~\eqref{I5-1} yields 
		\begin{align}\label{ccc16}
		I_5 = \,& \alpha \langle | \tilde{\nabla}_{h, (4)} \hat{\underline{\m}}_h^{\ell+3} |^2 
	\hat{\e}_h^{\ell+3} , 
	-\Delta_{h, (4)} ( 2 \tilde{\e}_h^{\ell+3} - \tilde{\e}_h^{\ell+2} )  \rangle  \\
		\le \, & \alpha  \Big\| | \tilde{\nabla}_{h, (4)} \hat{\underline{\m}}_h^{\ell+3} |^2 
 	\hat{\e}_h^{\ell+3} \Big\|_2   
	 \cdot \| \Delta_{h, (4)} ( 2 \tilde{\e}_h^{\ell+3} - \tilde{\e}_h^{\ell+2} ) \|_2 \nonumber\\
		\le \, & {\mathcal C}  \alpha \| \hat{\e}_h^{\ell+3} \|_2 
		 \cdot \| \Delta_{h, (4)} ( 2 \tilde{\e}_h^{\ell+3} - \tilde{\e}_h^{\ell+2} ) \|_2  
		 \nonumber \\ 
		 \le \, &	
		 {\mathcal C}  (   \| \e_h^{\ell+2} \|_2^2 
		 + \| \e_h^{\ell+1} \|_2^2  + \| \e_h^{\ell} \|_2^2 ) 		 
		 + \frac{\gamma_0}{16} ( \| \Delta_{h, (4)} \tilde{\e}_h^{\ell+3} \|_2^2
		  + \| \Delta_{h, (4)} \tilde{\e}_h^{\ell+2} \|_2^2 ) . \nonumber
		\end{align}
		
		\item Estimate $I_6$: A substitution of the preliminary estimate~\eqref{I6-1} leads to  		
		\begin{align}\label{point21}
		I_6=\,& \alpha \Big\langle \Big( 
	\tilde{\nabla}_{h, (4)}  ( \hat{\underline{\m}}_h^{\ell+3} + \hat{\m}_h^{\ell+3} ) 
	\cdot \tilde{\nabla}_{h, (4)} \hat{\e}_h^{\ell+3}  \Big) 		 
	\hat{\m}_h^{\ell+3} ,  
	-\Delta_{h, (4)} ( 2 \tilde{\e}_h^{\ell+3} - \tilde{\e}_h^{\ell+2} )  \Big\rangle   \\
		\le \, & \alpha \Big\| \Big( 
	\tilde{\nabla}_{h, (4)}  ( \hat{\underline{\m}}_h^{\ell+3} + \hat{\m}_h^{\ell+3} ) 
	\cdot \tilde{\nabla}_{h, (4)} \hat{\e}_h^{\ell+3}  \Big) 		 
	\hat{\m}_h^{\ell+3} \Big\|_2   \nonumber \\
	& \quad 
	\cdot \| \Delta_{h, (4)} ( 2 \tilde{\e}_h^{\ell+3} - \tilde{\e}_h^{\ell+2} ) \|_2  \nonumber\\
		\le \, & {\mathcal C}  \alpha \| \nabla_h  \hat{\e}_h^{\ell+3} \|_2 		
		 \cdot \| \Delta_{h, (4)} ( 2 \tilde{\e}_h^{\ell+3} - \tilde{\e}_h^{\ell+2} ) \|_2   \nonumber \\ 
		 \le 	\, & {\mathcal C}  (  \| \nabla_h \e_h^{\ell+2} \|_2^2 
		 + \| \nabla_h \e_h^{\ell+1} \|_2^2  
		 + \| \nabla_h \e_h^{\ell} \|_2^2 )  \nonumber \\
		 & \quad 
		 + \frac{\gamma_0}{16} ( \| \Delta_{h, (4)} \tilde{\e}_h^{\ell+3} \|_2^2
		  + \| \Delta_{h, (4)} \tilde{\e}_h^{\ell+2} \|_2^2 )    .	\nonumber
		\end{align}

	\end{itemize}
	Meanwhile, the inner product on the left hand side becomes the following identity, following similar telescope formula~\eqref{BDF-3-est-0}, combined with the summation by parts formula~\eqref{sum3}: 	
	\begin{align}  \label{ccc25}	
	L.H.S. = &  \frac{1}{k} \Big( 
	\| \alpha_1 \nabla_{h, (4)} \tilde{\e}_h^{\ell+3} \|_2^2 
	- \| \alpha_1 \nabla_{h, (4)} \tilde{\e}_h^{\ell+2} \|_2^2  
\\
  & \quad 
 + \|  \nabla_{h, (4)} ( \alpha_2 \tilde{\e}_h^{\ell+3} + \alpha_3 \tilde{\e}_h^{\ell+2} ) \|_2^2
 - \| \nabla_{h, (4)} ( \alpha_2 \tilde{\e}_h^{\ell+2} 
  + \alpha_3 \tilde{\e}_h^{\ell+1} ) \|_2 ^2  \nonumber 
\\
  & \quad 
  + \|  \nabla_{h, (4)} ( \alpha_4 \tilde{\e}_h^{\ell+3} + \alpha_5 \tilde{\e}_h^{\ell+2} 
  + \alpha_6 \tilde{\e}_h^{\ell+1} ) \|_2^2 \nonumber 
\\
   & \quad 
 - \|  \nabla_{h, (4)} ( \alpha_4 \tilde{\e}_h^{\ell+2} + \alpha_5 \tilde{\e}_h^{\ell+1} 
  + \alpha_6 \tilde{\e}_h^{\ell} ) \|_2^2 \nonumber
\\
  & \quad 
    + \|  \nabla_{h, (4)} ( \alpha_7 \tilde{\e}_h^{\ell+3} + \alpha_8 \tilde{\e}_h^{\ell+2} 
    + \alpha_9 \tilde{\e}_h^{\ell+1}
   + \alpha_{10} \tilde{\e}_h^{\ell} ) \|_2^2  \Big) .\nonumber
	\end{align}
	Substituting~\cref{ccc11}, \cref{ccc14,ccc16,I4-2,point21,ccc23} into~\cref{ccc73}, combined with~\cref{ccc25}, we arrive at
	\begin{align}\label{26}
	&   \| \alpha_1 \nabla_{h, (4)} \tilde{\e}_h^{\ell+3} \|_2^2 
	- \| \alpha_1 \nabla_{h, (4)} \tilde{\e}_h^{\ell+2} \|_2^2  
\\
  & 
 + \|  \nabla_{h, (4)} ( \alpha_2 \tilde{\e}_h^{\ell+3} + \alpha_3 \tilde{\e}_h^{\ell+2} ) \|_2^2
 - \| \nabla_{h, (4)} ( \alpha_2 \tilde{\e}_h^{\ell+2} 
  + \alpha_3 \tilde{\e}_h^{\ell+1} ) \|_2 ^2  \nonumber 
\\
  & 
  + \|  \nabla_{h, (4)} ( \alpha_4 \tilde{\e}_h^{\ell+3} + \alpha_5 \tilde{\e}_h^{\ell+2} 
  + \alpha_6 \tilde{\e}_h^{\ell+1} ) \|_2^2 \nonumber 
\\
   & 
 - \|  \nabla_{h, (4)} ( \alpha_4 \tilde{\e}_h^{\ell+2} + \alpha_5 \tilde{\e}_h^{\ell+1} 
  + \alpha_6 \tilde{\e}_h^{\ell} ) \|_2^2 \nonumber 
\\
  & 
    + \frac{\alpha}{2} k ( \| \Delta_{h, (4)} \tilde{\e}_h^{\ell+3} \|_2^2  
	          - \| \Delta_{h, (4)} \tilde{\e}_h^{\ell+2} \|_2^2  )   \nonumber 
\\ 
	& 
	+ ( \alpha - \frac{7 \beta_1}{2} - \frac{\gamma_0}{8} ) k 
	\| \Delta_{h, (4)} \tilde{\e}_h^{\ell+3} \|_2^2 
	- ( \frac52 \beta_1 + \frac{\gamma_0}{8} )  k 
	\| \Delta_{h, (4)} \tilde{\e}_h^{\ell+2} \|_2^2  \nonumber 
\\ 
	& 
	-  \beta_1  k 
	\| \Delta_{h, (4)} \tilde{\e}_h^{\ell+1} \|_2^2   
	+  ( \frac{\alpha}{2} - \frac74 \beta_1  ) k 
	\| \Delta_{h, (4)} ( \tilde{\e}_h^{\ell+3} - \tilde{\e}_h^{\ell+2} ) \|_2^2  \nonumber 
\\
       & 
	- \beta_1  k 
	\| \Delta_{h, (4)} ( \tilde{\e}_h^{\ell+2} - \tilde{\e}_h^{\ell+1} ) \|_2^2	 
     -  \frac{3 \beta_1}{4}  k 
	\| \Delta_{h, (4)} ( \tilde{\e}_h^{\ell+1} - \tilde{\e}_h^{\ell} ) \|_2^2	   \nonumber 	
\\
	\le \, & \mathcal{C} k \Big( \|\nabla_h \tilde{\e}_h^{\ell+3}\|_2^2 
	+ \| \nabla_h \tilde{\e}_h^{\ell+2}\|_2^2 + \|\nabla_h\e_h^{\ell+2}\|_2^2 	
	+ \|\nabla_h\e_h^{\ell+1}\|_2^2 +\|\nabla_h\e_h^{\ell}\|_2^2  \nonumber 
\\
   & \quad 
	+ \| \e_h^{\ell+2} \|_2^2 + \|\e_h^{\ell+1} \|_2^2 + \| \e_h^\ell \|_2^2 \Big) 
	 +\mathcal{C}k (k^6 + h^8) . 
	\nonumber
	\end{align}
	As a consequence, a combination of~\cref{ccc34} and~\cref{26} yields
	\begin{align}\label{convergence-1}
		& \| \alpha_1 \tilde{\e}_h^{\ell+3} \|_2^2 - \| \alpha_1 \tilde{\e}_h^{\ell+2} \|_2^2
 + \| \alpha_2 \tilde{\e}_h^{\ell+3} + \alpha_3 \tilde{\e}_h^{\ell+2} \|_2^2
 - \| \alpha_2 \tilde{\e}_h^{\ell+2} + \alpha_3 \tilde{\e}_h^{\ell+1} \|_2 ^2 
\\
  & 
  + \| \alpha_4 \tilde{\e}_h^{\ell+3} + \alpha_5 \tilde{\e}_h^{\ell+2} 
   + \alpha_6 \tilde{\e}_h^{\ell+1} \|_2^2
 - \| \alpha_4 \tilde{\e}_h^{\ell+2} + \alpha_5 \tilde{\e}_h^{\ell+1} 
 + \alpha_6 \tilde{\e}_h^{\ell} \|_2^2  \nonumber 	 
\\
  & 
      +  \| \alpha_1 \nabla_{h, (4)} \tilde{\e}_h^{\ell+3} \|_2^2 
	- \| \alpha_1 \nabla_{h, (4)} \tilde{\e}_h^{\ell+2} \|_2^2   \nonumber 
\\
  &  
 + \|  \nabla_{h, (4)} ( \alpha_2 \tilde{\e}_h^{\ell+3} + \alpha_3 \tilde{\e}_h^{\ell+2} ) \|_2^2
 - \| \nabla_{h, (4)} ( \alpha_2 \tilde{\e}_h^{\ell+2} 
  + \alpha_3 \tilde{\e}_h^{\ell+1} ) \|_2 ^2  \nonumber 
\\
  &  
  + \|  \nabla_{h, (4)} ( \alpha_4 \tilde{\e}_h^{\ell+3} + \alpha_5 \tilde{\e}_h^{\ell+2} 
  + \alpha_6 \tilde{\e}_h^{\ell+1} ) \|_2^2 \nonumber 
\\
   &  
 - \|  \nabla_{h, (4)} ( \alpha_4 \tilde{\e}_h^{\ell+2} + \alpha_5 \tilde{\e}_h^{\ell+1} 
  + \alpha_6 \tilde{\e}_h^{\ell} ) \|_2^2 \nonumber   
\\
  & 
    + \frac{\alpha}{2} k ( \| \nabla_{h, (4)} \tilde{\e}_h^{\ell+3} \|_2^2
                  + \| \Delta_{h, (4)} \tilde{\e}_h^{\ell+3} \|_2^2   
                  - \| \nabla_{h, (4)} \tilde{\e}_h^{\ell+2} \|_2^2                      
	          - \| \Delta_{h, (4)} \tilde{\e}_h^{\ell+2} \|_2^2  )   \nonumber 
\\ 
	& 
	+ ( \alpha - \frac{7 \beta_1}{2} - \frac{\gamma_0}{8} ) k 
	\| \Delta_{h, (4)} \tilde{\e}_h^{\ell+3} \|_2^2 
	- ( \frac52 \beta_1 + \frac{\gamma_0}{8} )  k 
	\| \Delta_{h, (4)} \tilde{\e}_h^{\ell+2} \|_2^2  \nonumber 
\\ 
	& 
	-  \beta_1  k 
	\| \Delta_{h, (4)} \tilde{\e}_h^{\ell+1} \|_2^2   
	+  ( \frac{\alpha}{2} - \frac74 \beta_1  ) k 
	\| \Delta_{h, (4)} ( \tilde{\e}_h^{\ell+3} - \tilde{\e}_h^{\ell+2} ) \|_2^2  \nonumber 
\\
       & 
	- \beta_1  k 
	\| \Delta_{h, (4)} ( \tilde{\e}_h^{\ell+2} - \tilde{\e}_h^{\ell+1} ) \|_2^2	 
     -  \frac{3 \beta_1}{4}  k 
	\| \Delta_{h, (4)} ( \tilde{\e}_h^{\ell+1} - \tilde{\e}_h^{\ell} ) \|_2^2	   \nonumber 	          	          
\\ 	
	\le \, & \mathcal{C} k \Big( \|\nabla_h\tilde{\e}_h^{\ell+3}\|_2^2 
	+ \|\nabla_h\tilde{\e}_h^{\ell+2}\|_2^2  
	+ \| \nabla_h \e_h^{\ell+2} \|_2^2 		
	+ \| \nabla_h \e_h^{\ell+1} \|_2^2 + \| \nabla_h \e_h^{\ell}\|_2^2 \nonumber \\ 
        & \quad 	
	+ \|\tilde{\e}_h^{\ell+3} \|_2^2 + \|\tilde{\e}_h^{\ell+2} \|_2^2 
	+ \| \e_h^{\ell+2} \|_2^2 + \| \e_h^{\ell+1}\|_2^2 + \| \e_h^{\ell}\|_2^2 \Big) 
	+ \mathcal{C} k (k^6 + h^8) . 
	\nonumber
	\end{align}	
	At this point, recalling the $W_h^{1,\infty}$ bound for $\m_h^k$ and $\tilde{\m}_h^k$, as given by~\eqref{bound-3}, \cref{bound-4}, and applying~\eqref{lem 6-2} in Lemma~\ref{lem 6-0}, we obtain
	\begin{equation*}
	\| \e_h^k \|_2 \le 2 \| \tilde{\e}_h^k \|_2  ,  \, \, \,
	\| \nabla_h \e_h^k \|_2 \le \mathcal{C} ( \| \nabla_h \tilde{\e}_h^k \|_2
	+ \| \tilde{\e}_h^k \|_2 ) , \quad \mbox{$\ell \le k \le \ell+2$} .  \label{convergence-2}
	\end{equation*}
	Its substitution into~\cref{convergence-1} leads to
	\begin{align} \label{convergence-3}
			& 
	 {\mathcal H}^{\ell+3}  - {\mathcal H}^{\ell+2}    
	 	+ ( \alpha - \frac{7 \beta_1}{2} - \frac{\gamma_0}{8} ) k 
	\| \Delta_{h, (4)} \tilde{\e}_h^{\ell+3} \|_2^2 
	- ( \frac52 \beta_1 + \frac{\gamma_0}{8} )  k 
	\| \Delta_{h, (4)} \tilde{\e}_h^{\ell+2} \|_2^2  
\\ 
	& 
	-  \beta_1  k 
	\| \Delta_{h, (4)} \tilde{\e}_h^{\ell+1} \|_2^2   
	+  ( \frac{\alpha}{2} - \frac74 \beta_1  ) k 
	\| \Delta_{h, (4)} ( \tilde{\e}_h^{\ell+3} - \tilde{\e}_h^{\ell+2} ) \|_2^2  \nonumber 
\\
       & 
	- \beta_1  k 
	\| \Delta_{h, (4)} ( \tilde{\e}_h^{\ell+2} - \tilde{\e}_h^{\ell+1} ) \|_2^2	 
     -  \frac{3 \beta_1}{4}  k 
	\| \Delta_{h, (4)} ( \tilde{\e}_h^{\ell+1} - \tilde{\e}_h^{\ell} ) \|_2^2	   \nonumber 		 
\\ 	
	\le \, & \mathcal{C} k \Big( \|\nabla_h \e_h^{\ell+3}\|_2^2  
	+ \| \nabla_h \e_h^{\ell+2} \|_2^2 		
	+ \| \nabla_h \e_h^{\ell+1} \|_2^2 + \| \nabla_h \e_h^{\ell}\|_2^2 \nonumber \\ 
        & \quad 	
	+ \| \e_h^{\ell+3} \|_2^2 + \| \e_h^{\ell+2} \|_2^2 
	+ \| \e_h^{\ell+1}\|_2^2 + \| \e_h^{\ell}\|_2^2 \Big) 
	+ \mathcal{C} k (k^6 + h^8) ,  \nonumber 
      \end{align} 
      
      \begin{align} 
  & 
  {\mathcal H}^{\ell+3} := \| \alpha_1 \tilde{\e}_h^{\ell+3} \|_2^2 
 + \| \alpha_2 \tilde{\e}_h^{\ell+3} + \alpha_3 \tilde{\e}_h^{\ell+2} \|_2^2    
 \\
   & \qquad 
  + \| \alpha_4 \tilde{\e}_h^{\ell+3} + \alpha_5 \tilde{\e}_h^{\ell+2} 
   + \alpha_6 \tilde{\e}_h^{\ell+1} \|_2^2   
  \nonumber 
\\
  &   \qquad 
       + \| \alpha_1 \nabla_{h, (4)} \tilde{\e}_h^{\ell+3} \|_2^2 
     + \|  \nabla_{h, (4)} ( \alpha_2 \tilde{\e}_h^{\ell+3} 
     + \alpha_3 \tilde{\e}_h^{\ell+2} ) \|_2^2 \nonumber  
\\
  &  \qquad 
  + \|  \nabla_{h, (4)} ( \alpha_4 \tilde{\e}_h^{\ell+3} + \alpha_5 \tilde{\e}_h^{\ell+2} 
  + \alpha_6 \tilde{\e}_h^{\ell+1} ) \|_2^2  \nonumber  
\\
  &  \qquad 
    + \frac{\alpha}{2} k ( \| \nabla_{h, (4)} \tilde{\e}_h^{\ell+3} \|_2^2
                  + \| \Delta_{h, (4)} \tilde{\e}_h^{\ell+3} \|_2^2  ) .   
	\nonumber
	\end{align} 
In turn, a summation in time reveals that 
\begin{equation} 
\begin{aligned} 
   {\mathcal H}^{\ell+3}  
   + \frac{\gamma_0}{4} k \sum_{j=0}^{\ell+3} \| \Delta_{h, (4)} \tilde{\e}_h^j \|_2^2   
   \le & {\mathcal C} k \sum_{j=0}^{\ell+3}    \Big( \|\nabla_h \e_h^j \|_2^2  
	 + \| \e_h^j \|_2^2  \Big) 
	+ \mathcal{C} T (k^6 + h^8) 
\\
  \le & 
  {\mathcal C} k \sum_{j=0}^{\ell+3} {\mathcal H}^j  
	+ \mathcal{C} T (k^6 + h^8) ,  
\end{aligned} 
\end{equation} 
in which the estimate that $\| \tilde{\e}_h^j \|_2^2 , \| \nabla_h \tilde{\e}_h^j \|_2^2 \le {\mathcal C} {\mathcal H}^j$, as well as the following fact, have been applied: 
\begin{equation*} 
\begin{aligned} 
 & 
 \alpha - \frac{7 \beta_1}{2} - \frac{\gamma_0}{8}  - \frac52 \beta_1 - \frac{\gamma_0}{8} -  \beta_1  = \alpha - 7 \beta_1 - \frac{\gamma_0}{4} 
  = \frac{\gamma_0}{4} >0 , 
\\
  & 
  \frac{\alpha}{2} - \frac74 \beta_1 - \beta_1 - \frac43 \beta_1 = \frac{\alpha}{2} - \frac72 \beta_1 
  = \frac{\gamma_0}{4} >0  , 
\end{aligned} 
\end{equation*} 
Therefore, an application of discrete Gronwall inequality (in Lemma~\ref{ccclem1}) yields the desired convergence estimate for $\tilde{\e}_h$:
	\begin{align*}
	\| \tilde{\e}_h^n\|_2^2 + \|\nabla_h\tilde{\e}_h^n\|_2^2  
	\le {\mathcal H}^n 
	 \leq \mathcal{C}Te^{\mathcal{C}T}(k^6 + h^8 ), \quad \forall \, n\leq \left\lfloor\frac{T}{k}\right\rfloor , 
	\end{align*}
	i.e.,
	\begin{align*} 
	\|\tilde{\e}_h^n\|_2 + \|\nabla_h\tilde{\e}_h^n\|_2
	&\le \mathcal{C}(k^3 + h^4) .
	\end{align*}
	An application of Lemma~\ref{ccclemC1}, as well as the time step constraint $k\leq \mathcal{C}h$, leads to
	\begin{align}\label{bound-6}
	\|\tilde{\e}_h^n\|_{\infty} &\leq \frac{\|\tilde{\e}_h^n\|_2}{h^{d/2}}\leq 
	\frac{\mathcal{C} (k^3 + h^4)}{h^{d/2}}\leq \frac{1}{6}, \\
	\|\nabla_h\tilde{\e}_h^n\|_{\infty} &\leq 
	\frac{\|\nabla_h\tilde{\e}_h^n\|_2}{h^{d/2}} 
	\leq \frac{\mathcal{C} (k^3 + h^4)}{h^{d/2}}\leq \frac{1}{6},\nonumber
	\end{align}
	so that the second part of the \textit{a priori} assumption~\cref{bound-2} has been recovered at time step $k=n$. In turn, the $W_h^{1,\infty}$ bound~\cref{bound-4} becomes available, which enables us to apply~\cref{lem 6-2} in Lemma~\ref{lem 6-0}, and obtain the desired convergence estimate for $\e_h^n$:
	\begin{align*}
	& \| \e_h^n \|_2 \le 2 \| \tilde{\e}_h^n \|_2 
	\le \mathcal{C} (k^3 + h^4)  ,  \label{convergence-6}  \\
	& \| \nabla_h \e_h^n \|_2 \le \mathcal{C} ( \| \nabla_h \tilde{\e}_h^n \|_2
	+ \| \tilde{\e}_h^n \|_2 ) 
	\le \mathcal{C} (k^3 + h^4) .  \nonumber
	\end{align*}
	Similar to the derivation of~\cref{bound-6}, we also get
	\begin{equation}  
\begin{aligned} 
	\| \e_h^n\|_{\infty} \le  & {\mathcal C} ( \| \e_h^n\|_{H_h^1} 
	+ \| \e_h^n\|_{H_h^1}^\frac34 \cdot \| \nabla_h \Delta_h \e_h^n\|_2^\frac14 ) 
\\
  \le & 
       {\mathcal C} \Big( \| \e_h^n\|_{H_h^1} 
	+ \| \e_h^n\|_{H_h^1}^\frac34 \cdot \Big( \frac{\| \nabla_h \e_h^n \|_2}{h^2} \Big)^\frac14 \Big)    
\\
  \le & 
   \frac{ {\mathcal C} ( k^3 + h^4 )}{h^\frac12}  \le \dt + h ,  
\\ 
	\|\nabla_h \e_h^n\|_{\infty} \le & \frac{1}{6} , 
	\quad \mbox{(similar derivation as~\eqref{bound-6}} , 	
\end{aligned}  
\label{bound-7}
	\end{equation} 
provided that $k \le {\mathcal C} h$, and $k$ and $h$ are sufficiently small. We also notice that the first inequality in~\eqref{bound-7} stands for a discrete Gagliardo-Nirenberg inequality in the finite difference version, which has been derived in a recent work~\cite{chen16}. As a consequence, the first part of the \textit{a priori} assumption~\cref{bound-2} has been recovered at time step $k=\ell+3$. This completes the proof of Theorem~\ref{cccthm2}.
\end{proof}

\begin{remark} 
The regularity assumption for the exact solution, namely $\m_e \in C^4 ([0,T]; C^0) \cap C^3([0,T]; C^1) \cap L^{\infty}([0,T]; C^6)$, as stated in Theorem~\ref{cccthm2}, is very strong. In fact, a global-in-time weak solution of the LL equation~\cref{c1} is only of regularity class $L^{\infty}([0,T]; H^1) \cap L^2([0,T]; H^3)$. Of course, if the initial data is smooth enough, one could always derive a local-in-time exact solution with higher enough regularity estimate, so that the convergence estimate established in Theorem~\ref{cccthm2} could pass through. In other words, the optimal rate error estimate~\eqref{convergence-0} stands for a local-in-time theoretical result. In addition, since the finite difference numerical method is evaluated at the collocation grid points, instead of the ones based on a weak formulation, it usually requires higher order regularity requirement for the exact solution in the optimal rate convergence estimate than that of the finite element approach; see the related finite difference analysis for various gradient flows~\cite{baskaran13b, wang11a, wise09a}, etc.    
\end{remark} 

\begin{remark} 
The condition $\alpha > 7$ is a very strong constraint. In fact, such a condition is used in the estimate~\eqref{ccc11} for $I_1$, we need $\alpha > 7$ to control these Laplacian terms, due to the explicit treatment of $\Delta_h \hat{\tilde{\m}}_h^{\ell+3}$. Meanwhile, such an inequality only stands for a theoretical difficulty, and the practical computations may not need that large value of $\alpha$. In most practical simulation examples, a value of $\alpha > 1$ would be sufficient to ensure thee numerical stability of the proposed numerical scheme~\cref{cc}-\cref{scheme-1-2}. 

In addition, the explicit treatment of the Laplacian term, namely $\Delta_{h, (4)} \hat{\tilde{\m}}_h^{n+3} = \Delta_{h, (4)} ( 3 \tilde{\m}_h^{n+2} - 3 \tilde{\m}_h^{n+1} + \tilde{\m}_h^n )$, will greatly improve the numerical efficiency, since only a constant-coefficient Poisson solver is needed ta each step. This crucial fact enables one to produce very robust numerical simulation results at a much-reduced computational cost.  
\end{remark}

\section{Numerical examples}
\label{sec:experiments}

In this section, we verify its accuracy in one-dimentional (1D) and three-dimentional (3D) cases. In 1D, we choose the exact solution as below,
\begin{align*}
\m_e=[\cos(\cos(\pi x)) \sin (t), \sin(\cos(\pi x)) \sin(t),\cos(t)].
\end{align*}
	
The spatial accuracy in 1D is shown in \cref{spaceAccuracy-v2}. The temporal accuracy in 1D \cref{timeAccuracy-v2}. 
\begin{table}[htbp]
	\centering
	\caption{Spatial accuracy for our proposed scheme in 1D with $\alpha=10$, $N_t=1$e5.}\label{spaceAccuracy-v2}
	\begin{tabular}{c|c|c|c}
		\hline
		$h$&$\|\m_h-\m_e\|_{\infty}$ &$\|\m_h-\m_e\|_2$ &$\|\m_h-\m_e\|_{H^1}$ \\
		\hline
		$1/16$&7.725597545818474e-06& 5.836998359249282e-06& 9.863379588342884e-05  \\
		\hline
		$1/32$& 5.043991847669682e-07& 3.708268012068762e-07& 6.357960049137558e-06   \\
		\hline
		$1/64$& 3.188098195161526e-08& 2.327672256026284e-08& 4.005750052425590e-07 \\
		\hline
		$1/128$&1.998533172287154e-09& 1.456447923591758e-09& 2.508658656526442e-08  \\
		\hline
		$1/256$&1.248005865317481e-10& 9.109844749647466e-11& 1.568595885318644e-09 \\
		\hline
		$1/512$&1.267810856298013e-11& 9.094912356564941e-12& 9.626833899837611e-11  \\
		\hline
		order &3.89		 &3.90&3.99\\
		\hline
	\end{tabular}
\end{table}

\begin{table}[htbp]
	\centering
	\caption{Temporal accuracy for our proposed scheme in 1D with $\alpha=10$, $N_x=1$e4.}\label{timeAccuracy-v2}
	\begin{tabular}{c|c|c|c}
		\hline
		$k$&$\|\m_h-\m_e\|_{\infty}$ &$\|\m_h-\m_e\|_2$ &$\|\m_h-\m_e\|_{H^1}$ \\
		\hline
		$T/8$ & 1.981641473136619e-07& 1.387988010512471e-07& 6.165376490036942e-07 \\
		\hline
		$T/12$& 5.829046183930542e-08& 4.227466726190239e-08& 1.880848678938826e-07\\
		\hline
		$T/16$& 2.484947866226994e-08& 1.712105680469501e-08& 7.576681980972073e-08 \\
		\hline
		$T/24$& 7.528568560233317e-09& 4.909750409830862e-09& 2.140936685691737e-08\\
		\hline
		$T/32$ &3.027002620781261e-09& 2.286226144082027e-09& 1.018690881652361e-08\\
		\hline
		order &3.00		 &2.99&3.00\\
		\hline
	\end{tabular}
\end{table}


In 3D, we take the exact solution as below,
\begin{align*}
\m_e=[\cos(\cos(\pi x)\cos(\pi y)\cos(\pi z)) \sin (t), \sin(\cos(\pi x)\cos(\pi y)\cos(\pi z)) \sin(t),\cos(t)].
\end{align*}

The results for spatial accuracy are presented in \cref{spaceAccuracy-A-2}. The temporal accuracy is shown in \cref{timeAccuracy-A-4}.


\begin{table}[htbp]
	\centering
	\caption{Spatial accuracy for our proposed scheme in 3D with $\alpha=10$, $N_t=1$e4 and $T=1$.}\label{spaceAccuracy-A-2}
	\begin{tabular}{c|c|c|c}
		\hline
		$h$&$\|\m_h-\m_e\|_{\infty}$ &$\|\m_h-\m_e\|_2$ &$\|\m_h-\m_e\|_{H^1}$ \\
		\hline
		$1/12$& 0.482160540597345& 0.281542274274310 &0.393105651737428\\
		\hline
		$1/16$& 0.148299462755549 & 0.092398029798804 &0.131951193056456    \\
		\hline
		$1/20$& 0.059102162889618 & 0.037761936069616 &0.054265772634242  \\
		\hline
		$1/24$& 0.028136229626891 & 0.018143447408694 &0.026115549952996  \\
		\hline
		$1/28$&0.015083789029621 & 0.009766140937128 &0.014061628506293   \\
		\hline
		order &4.09		 &3.97&3.94\\
		\hline
	\end{tabular}
\end{table}


\begin{table}[htbp]
	\centering
	\caption{Temporal accuracy for our proposed scheme in 3D with $\alpha=10$, and $T=1$.}\label{timeAccuracy-A-4}
	\begin{tabular}{c|c|c|c|c}
		\hline
		$h$&$k,\;k^3\approx h^4$&$\|\m_h-\bm_e\|_{\infty}$ &$\|\m_h-\m_e\|_2$ &$\|\m_h-\m_e\|_{H^1}$ \\
		\hline
		$1/16$&$1/40$& 0.232983042019129 & 0.142854067479002 &0.201849553240457\\
		\hline
		$1/20$&$1/54$& 0.109240092920532 & 0.068926003736393  &0.098135067448825  \\
		\hline
		$1/24$&$1/69$& 0.055967639432972 & 0.035819688414051 & 0.051231145694668 \\
		\hline
		$1/28$&$1/85$&0.031061236247365 & 0.020021764764969 & 0.028699857658575  \\
		\hline
		$1/32$&$1/101$&0.018690676252636 & 0.012092732078072 & 0.017350944169495  \\
		\hline
		order& &2.72		 &2.67&2.65\\
		\hline
	\end{tabular}
\end{table}

\section{Conclusions}
\label{sec:conclusions}

In this paper, we develop a fully discrete finite difference scheme for the LLG equation, with the fourth order spatial accuracy and third order temporal accuracy. The fourth order spatial accuracy is obtained by a long stencil finite difference, and a symmetric boundary extrapolation is applied, based on a higher order Taylor's expansion around the boundary section. The third-order backward differentiation formula is applied in the temporal discretization, the linear diffusion term is treated implicitly, while the nonlinear terms are updated by a fully explicit extrapolation formula. A detailed convergence analysis and error estimate are provided for the proposed numerical scheme, which gives an optimal $\mathcal{O}(k^3+h^4)$ accuracy order in the $\ell^{\infty}([0,T];\ell^2)\cup \ell^2([0,T];H_h^1)$ norm under suitable regularity assumptions and reasonable ratio between the time step-size and the spatial mesh-size. Numerical examples are presented to verify its theoretical analysis.


\section*{Acknowledgments}
We thank Zheyu Xia from University of Electronic Science and Technology of China for helpful discussions.
This work is supported in part by the NSF DMS-2012669 (C.~Wang), and Jiangsu Science and Technology Programme-Fundamental Research Plan Fund, Research and Development Fund of XJTLU (RDF-24-01-015) (C.~Xie).

\bibliographystyle{amsplain}
\bibliography{draft2_BDF3}

\begin{bibdiv}
\begin{biblist}

\bib{bartels2006convergence}{article}{
      author={Bartels, S.},
      author={Prohl, A.},
       title={Convergence of an implicit finite element method for the
  {Landau-Lifshitz-Gilbert} equation},
        date={2006},
     journal={SIAM J. Numer. Anal.},
      volume={44},
      number={4},
       pages={1405\ndash 1419},
}

\bib{baskaran13b}{article}{
      author={Baskaran, A.},
      author={Lowengrub, J.},
      author={Wang, C.},
      author={Wise, S.},
       title={Convergence analysis of a second order convex splitting scheme
  for the modified phase field crystal equation},
        date={2013},
     journal={SIAM J. Numer. Anal.},
      volume={51},
       pages={2851\ndash 2873},
}

\bib{Cai2022JCP}{article}{
      author={Cai, Y.},
      author={Chen, J.},
      author={Wang, C.},
      author={Xie, C.},
       title={A second-order numerical method for {Landau-Lifshitz-Gilbert}
  equation with large damping parameters},
        date={2022},
        ISSN={0021--9991},
     journal={J. Comput. Phys.},
      volume={451},
       pages={110831},
}

\bib{Cai2023MMAS}{article}{
      author={Cai, Y.},
      author={Chen, J.},
      author={Wang, C.},
      author={Xie, C.},
       title={Error analysis of a linear numerical scheme for the
  {Landau-Lifshitz} equation with large damping parameters},
        date={2023},
     journal={Math. Methods Appl. Sci.},
      volume={46},
      number={18},
       pages={18952\ndash 18974},
}

\bib{chen20}{article}{
      author={Chen, J.},
      author={Wang, C.},
      author={Xie, C.},
       title={Convergence analysis of a second-order semi-implicit projection
  method for {Landau-Lifshiz} equation},
        date={2020},
     journal={Appl. Numer. Math.},
        note={Submitted and in review},
}

\bib{chen16}{article}{
      author={Chen, W.},
      author={Liu, Y.},
      author={Wang, C.},
      author={Wise, S.M.},
       title={An optimal-rate convergence analysis of a fully discrete finite
  difference scheme for {Cahn-Hilliard-Hele-Shaw} equation},
        date={2016},
     journal={Math. Comp.},
      volume={85},
       pages={2231\ndash 2257},
}

\bib{Ciarlet1978}{book}{
      author={Ciarlet, P.G.},
       title={{The Finite Element Method for Elliptic Problems}},
   publisher={Elsevier Science},
        date={1978},
}

\bib{cimrak2004iterative}{article}{
      author={Cimr{\'a}k, I.},
      author={Slodi{\v{c}}ka, M.},
       title={An iterative approximation scheme for the
  {Landau-Lifshitz-Gilbert} equation},
        date={2004},
     journal={J. Comput. Appl. Math.},
      volume={169},
      number={1},
       pages={17\ndash 32},
}

\bib{weinan2001numerical}{article}{
      author={E, W.},
      author={Wang, X.},
       title={Numerical methods for the {Landau-Lifshitz} equation},
        date={2001},
     journal={SIAM J. Numer. Anal.},
      volume={38},
       pages={1647\ndash 1665},
}

\bib{gao2014optimal}{article}{
      author={Gao, H.},
       title={Optimal error estimates of a linearized {Backward Euler FEM} for
  the {Landau-Lifshitz} equation},
        date={2014},
     journal={SIAM J. Numer. Anal.},
      volume={52},
      number={5},
       pages={2574\ndash 2593},
}

\bib{Girault1986}{book}{
      author={Girault, V.},
      author={Raviart, P.A.},
       title={{Finite Element Methods for Navier-Stokes Equations, Theorems and
  Algorithms }},
   publisher={Springer-Verlag},
        date={1986},
}

\bib{Hao2020}{article}{
      author={Hao, Y.},
      author={Huang, Q.},
      author={Wang, C.},
       title={A third order {BDF} energy stable linear scheme for the
  no-slope-selection thin film model},
        date={2020},
     journal={Commun. Comput. Phys.},
        note={Accepted and in press},
}

\bib{Landau2015On}{article}{
      author={Landau, L.D.},
      author={Lifshits, E.M.},
       title={On the theory of the dispersion of magnetic permeability in
  ferromagnetic bodies},
        date={1935},
     journal={Phys. Z. Sowjet.},
      volume={63},
      number={9},
       pages={153\ndash 169},
}

\bib{JLiu2013}{article}{
      author={Liu, J.},
       title={Simple and efficient ale methods with provable temporal accuracy
  up to fifth order for the stokes equations on time varying domains},
        date={2013},
     journal={SIAM J. Numer. Anal.},
      volume={51},
       pages={743\ndash 772},
}

\bib{STWW2003}{article}{
      author={Samelson, R.},
      author={Temam, R.},
      author={Wang, C.},
      author={Wang, S.},
       title={Surface pressure {Poisson} equation formulation of the primitive
  equations: {Numerical} schemes},
        date={2003},
     journal={SIAM J. Numer. Anal.},
      volume={41},
       pages={1163\ndash 1194},
}

\bib{Wang2000}{article}{
      author={Wang, C.},
      author={Liu, J.-G.},
       title={Convergence of gauge method for incompressible flow},
        date={2000},
     journal={Math. Comp.},
      volume={69},
       pages={1385\ndash 1407},
}

\bib{Wang2004}{article}{
      author={Wang, C.},
      author={Liu, J.-G.},
      author={Johnston, H.},
       title={Analysis of a fourth order finite difference method for
  incompressible {Boussinesq} equation},
        date={2004},
     journal={Numer. Math.},
      volume={97},
       pages={555\ndash 594},
}

\bib{wang11a}{article}{
      author={Wang, C.},
      author={Wise, S.M.},
       title={An energy stable and convergent finite-difference scheme for the
  modified phase field crystal equation},
        date={2011},
     journal={SIAM J. Numer. Anal.},
      volume={49},
       pages={945\ndash 969},
}

\bib{wise09a}{article}{
      author={Wise, S.M.},
      author={Wang, C.},
      author={Lowengrub, J.},
       title={An energy stable and convergent finite-difference scheme for the
  phase field crystal equation},
        date={2009},
     journal={SIAM J. Numer. Anal.},
      volume={47},
       pages={2269\ndash 2288},
}

\bib{yao17}{article}{
      author={Yao, C.},
      author={Wang, C.},
      author={Kou, Y.},
      author={Lin, Y.},
       title={A third order linearized {BDF} scheme for {Maxwell's} equations
  with nonlinear conductivity using finite element method},
        date={2017},
     journal={Int. J. Numer. Anal. Model.},
      volume={14},
       pages={511\ndash 531},
}

\end{biblist}
\end{bibdiv}
%
%
%
%

\end{document}